\documentclass[11pt]{amsart}
\usepackage[margin=1in]{geometry}

\usepackage{amssymb}
\usepackage{amsthm}
\usepackage{amsmath}
\usepackage{mathrsfs}
\usepackage{amsbsy}
\usepackage[all]{xy}
\usepackage{bm}
\usepackage{hyperref}
\usepackage{tikz}
\usetikzlibrary{quotes}
\usepackage{array}
\usepackage{enumerate}
\usepackage{xcolor}
\usepackage{bbm}
\usepackage{comment}
\usepackage{dynkin-diagrams}
\usepackage{ifthen}
\usepackage{thmtools, thm-restate}
\graphicspath{{./}{figures/}}

\usepackage[noabbrev,capitalise,nameinlink]{cleveref}

\definecolor{lavender}{rgb}{0.4,0,1.0}
\definecolor{peach}{rgb}{1,0.43,0.39}
\definecolor{DarkPink}{rgb}{0.9,0,0.45} 
\definecolor{Green}{RGB}{0,204,0} 

\hypersetup{colorlinks=true, citecolor=DarkPink, linkcolor=DarkPink,urlcolor=DarkPink}

\usepackage{hhline}
\allowdisplaybreaks
\usepackage[noadjust]{cite}
\usepackage{asymptote}

\usepackage{caption}
\usepackage{tabu}
\usepackage{diagbox}
\usepackage[noabbrev,capitalise,nameinlink]{cleveref}
\crefname{conjecture}{Conjecture}{Conjectures}

\newtheorem{mainthm}{Theorem}

\newtheorem{theorem}{Theorem}[section]
\newtheorem{proposition}[theorem]{Proposition}
\newtheorem{corollary}[theorem]{Corollary}

\newtheorem{lemma}[theorem]{Lemma}

\theoremstyle{definition}
\newtheorem{definition}[theorem]{Definition}
\newtheorem{remark}[theorem]{Remark}

\newcommand{\Inv}{\mathrm{Inv}} 
 
\newcommand{\Cay}{\mathrm{Cay}}

\newcommand{\snake}{\mathrm{snake}}

\newcommand{\bx}{\mathbf{x}}
\newcommand{\bu}{\mathbf{u}}
\newcommand{\bv}{\mathbf{v}}
\newcommand{\wh}{\widehat}

\newcommand{\height}[1]{\mathrm{height}(#1)}
\newcommand{\id}{\mathrm{id}} 
\newcommand{\yy}{\mathtt{src}}
\newcommand{\zz}{\texttt{tar}}
\newcommand{\qform}{Q}

\renewcommand{\AA}{\mathbb{A}}

\newcommand{\DD}{\mathbb{D}}

\newcommand{\RR}{\mathbb{R}}

\newcommand{\ZZ}{\mathbb{Z}}
\newcommand{\Span}{\mathrm{Span}}
\newcommand{\spange}{\Span_{\geq 0}}
\newcommand{\conv}{\mathrm{Conv}}

\newcommand{\cA}{\mathcal{A}}

\newcommand{\cH}{\mathcal{H}}

\newcommand{\cL}{\mathcal{L}}

\newcommand{\cQ}{\mathcal{Q}}

\newcommand{\cS}{\mathcal{S}}
\newcommand{\cT}{\mathcal{T}}
\newcommand{\cU}{\mathcal{U}}

\newcommand{\cZ}{\mathcal{Z}}

\newcommand{\dfn}[1]{\textcolor{lavender}{\emph{#1}}}

\begin{document}

\title[Extended Weak Order for $U_3$]{Extended Weak Order for the
  Rank~$3$\\ Universal Coxeter Group}
\subjclass[2020]{}

\author[]{Grant Barkley} \address[GB]{Department of Mathematics,
  University of Michigan, Ann Arbor, MI 48109, USA}
\email{gbarkley@umich.edu}

\author[]{Colin Defant} \address[CD]{Department of Mathematics,
  Harvard University, Cambridge, MA 02138, USA}
\email{colindefant@gmail.com}

\author[]{Patricia Hersh} \address[PH]{Department of Mathematics,
  University of Oregon, Eugene, OR 97403, USA}
\email{plhersh@uoregon.edu}

\author[]{Jon McCammond} \address[JM]{Department of Mathematics,
  University of California, Santa Barbara, CA 93106, USA}
\email{jon.mccammond@math.ucsb.edu}

\author[]{Thomas McConville} \address[TM]{Department of Mathematics,
  Kennesaw State University, Marietta, GA 30060, USA}
\email{tmcconvi@kennesaw.edu}

\author[]{David E Speyer} \address[DES]{Department of Mathematics,
  University of Michigan, Ann Arbor, MI 48109, USA}
\email{speyer@umich.edu} 

\thanks{GB was supported by NSF grants DMS-2152991 and
  DMS-2503536.  CD was supported by NSF grant DMS-2201907 and by a Benjamin Peirce Fellowship at Harvard University. PH was supported by Simons Foundation travel funding for mathematicians. JM was supported by NSF grant DMS-2204001.  DES was supported by NSF grant DMS-2246570.}

\begin{abstract} 
  The weak order is a classical poset structure on a Coxeter group; it is a lattice when the group is
  finite but merely a meet-semilattice when the group is infinite.
  Motivated by problems in Kazhdan--Lusztig theory, Matthew Dyer
  introduced the \emph{extended weak order}, a poset that
  contains a copy of the weak order as an order ideal, and he
  conjectured that the extended weak order for any Coxeter group is a
  lattice.  We prove Dyer's conjecture for the rank~$3$ universal
  Coxeter group.  This is the first non-spherical, non-affine Coxeter
  group for which Dyer's conjecture has been proven.
\end{abstract}

\maketitle

\section*{Introduction}\label{sec:intro}

The \emph{weak order} is a classical partial order on a Coxeter group $W$. A seminal result due to Bj\"orner \cite{BjornerLattice} states that the weak order on $W$ is a meet-semilattice. When $W$ is finite/spherical, the weak order is a
complemented lattice. In this case, the Hasse diagram of the weak order can be viewed as an orientation of the $1$-skeleton of the polytope known as the \emph{$W$-permutohedron}, whose face lattice is related to the lattice structure of the weak order.  When $W$ is infinite, the weak order is not a lattice. 

Let $\Phi^+$ denote the set of positive roots in the root system of $W$. A set $R\subseteq\Phi^+$ is \dfn{closed} if for all $\alpha,\beta\in
  R$, every root in the nonnegative span of $\alpha$ and $\beta$ is also in $R$. We say $R$ is \dfn{biclosed} if $R$ and $\Phi^+\setminus R$ are both closed. Motivated by problems in
Kazhdan--Lusztig theory, Dyer defined the \dfn{extended weak order} for $W$ to be the collection of biclosed subsets of $\Phi^+$, partially ordered by containment. This poset contains an order ideal that is isomorphic to the weak order. 

One of the major open conjectures about biclosed sets is Dyer's conjecture \cite[Conjecture~2.5(i)]{dyer:2019weak} that the extended weak order is a complete lattice, which is a vast strengthening of Bj\"orner's result. Dyer first raised this question in \cite[Remark~2.14]{dyer:1994quotients}; see also the discussion in \cite[Section~2.4]{dyer:2019weak}. Edgar also records the conjecture in \cite[Conjecture~9.2.4]{EdgarThesis}. For finite/spherical Coxeter groups, the conjecture is already known to hold because the weak order and extended weak order are isomorphic.  Wang \cite{wang:2019infinite} proved Dyer's conjecture for affine Coxeter groups of rank $3$, and Barkley and Speyer \cite{barkley:2023affine} proved it for all affine Coxeter groups.  In this article, we prove Dyer's conjecture for the universal Coxeter group of rank~$3$, which we
denote by $U_3$.  This is the first non-spherical, non-affine Coxeter group for which
Dyer's conjecture has been established. In fact, we prove a slightly stronger result. 

\begin{mainthm}\label{thm:main} 
  Every biclosed set of positive roots for $U_3$ is weakly
  separable.  As a consequence, the extended weak order for $U_3$ is a
  lattice.
\end{mainthm} 

Our techniques used to prove \cref{thm:main} are motivated by geometric group theory and are very different from those used to study the extended weak order in the past. We believe they will be useful for proving further cases of Dyer's conjecture in the future. A brief outline is as follows. We introduce the notion of a \emph{parabolic biclosed} subset of $\Phi^+$, which is a subset of $\Phi^+$ whose intersection with each rank 2 parabolic subsystem is biclosed in the set of positive roots of that subsystem. There is a natural bijection between $\Phi^+$ and the set of edges of the Cayley graph $\Gamma$ of $U_3$, which embeds in the hyperbolic plane. Given a bipartition $R\sqcup B$ of $\Phi^+$ into parabolic biclosed sets $R$ and $B$, we color the edges of $\Gamma$ corresponding to $B$ blue and color the edges corresponding to $R$ red. We prove that the red and blue edges can be separated by a connected topological $1$-manifold called a \emph{pair of snakes}, which consists of two \emph{snakes} slithering through the Cayley graph. The pair of snakes is homeomorphic to a circle when $R$ or $B$ is finite, and it is homeomorphic to  a line otherwise (see \cref{fig:finite}). We then prove
that when $R$ and $B$ are infinite and biclosed, the two heads of the snake define a hyperplane that weakly separates $R$ from $B$, establishing \cref{thm:main}.

\begin{figure}
  \includegraphics[height=7.3cm]{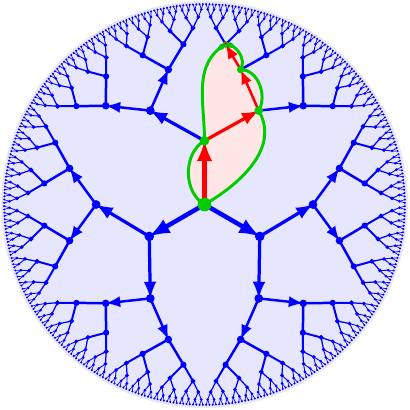}\quad\includegraphics[height=7.3cm]{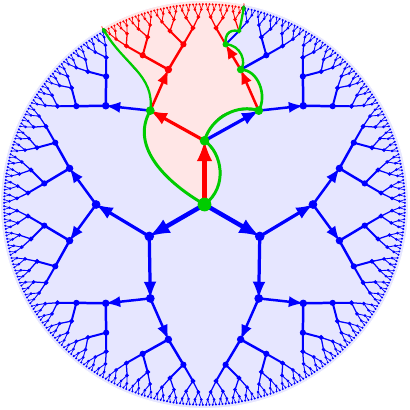}
  \caption{In each figure, a {\color{Green}green} pair of snakes separates the {\color{red} red} edges of $\Gamma$ corresponding to a parabolic biclosed set $R\subseteq\Phi^+$ from the {\color{blue}blue} edges corresponding to the set $B=\Phi^+\setminus R$. On the left, $R$ is finite. On the right, $R$ and $B$ are both infinite. 
    \label{fig:finite}} 
\end{figure}

Our second main result, which concerns the more refined lattice structure of the extended weak order, settles another conjecture due to Dyer for the universal Coxeter group of rank $3$ (see \cite[Conjecture~2.5]{dyer:2019weak}). 

\begin{mainthm}\label{thm:main2} 
  Let $\{R_i\}_{i\in I}$ be a collection of biclosed sets of positive roots for $U_3$. The join of $\{R_i\}_{i\in I}$ in the extended weak order is the smallest closed subset of $\Phi^+$ that contains $\bigcup_{i\in I}R_i$.  
\end{mainthm} 

The structure of the article is as follows.  The first three sections are largely preparatory. First, 
\cref{sec:hyperbolic} fixes the notation regarding hyperbolic geometry that will be used throughout the paper. Second, \cref{sec:universal} specializes the standard reflection representation, the Farey tessellation, and the dual Cayley graph to the rank $3$ universal Coxeter group $U_3$. Third, \cref{sec:invert-convex} recalls standard notions related to weak order and extended weak order, including inversion sets, biclosed sets, separable and weakly separable sets, normalized roots, and limit roots. The only new notion introduced in these preparatory sections is that of a \emph{parabolic biclosed} set in \cref{sec:invert-convex}.  
The new arguments begin in \cref{sec:parabolic}, where we introduce snakes and prove that the pair of snakes associated to a parabolic biclosed set $R$ is a topological $1$-manifold separating the edges corresponding to $R$ from the edges corresponding to $\Phi^+\setminus R$.  
In \cref{sec:biclosed}, we use the stronger hypothesis that $R$ is biclosed, rather than merely parabolic biclosed, to show that the two heads of the snake determine a hyperplane that weakly separates $R$ from $\Phi^+\setminus R$, thereby proving \cref{thm:main}.  
Then \cref{sec:When?} explains, in terms of jagged snakes, when a parabolic biclosed set is globally biclosed.  
\cref{sec:structure} describes more detailed lattice-theoretic properties of the extended weak order for $U_3$.  
Finally, \cref{sec:closure} proves \cref{thm:main2}, identifying arbitrary joins with the $2$-closure of the union.

\section{Hyperbolic geometry}\label{sec:hyperbolic}

This section reviews basic facts about hyperbolic geometry. 
For additional background see standard references such as \cite{Anderson05,CFP97,Humphreys}.
Our presentation is particularly influenced by \cite{DyerHohlwegRipoll} and \cite{bjorner:2005Coxbook}.
Let $\Delta=\{\alpha_1,\alpha_2,\alpha_3\}$ be a basis of $V=\RR^3$.
Let $\langle -,-\rangle$ be the bilinear form such that 
\[ 
\langle \alpha_i,\alpha_j\rangle = 
\begin{cases}
    1\ &\mbox{if }i=j\\
    -1\ &\mbox{if }i\neq j
\end{cases}.
\]
Then $\langle -,-\rangle$ is a nondegenerate symmetric bilinear form on $V$ of signature $(2,1)$.
The \dfn{indefinite orthogonal group} $O(2,1)$ is the group of linear transformations of $V$ that preserve the bilinear form.
For $\bx\in V$, there are unique real numbers $x,y,z$ with $\bx=x\alpha_1+y\alpha_2+z\alpha_3$.
We use the coordinates $(x,y,z)$ to represent $\bx$.

The quadratic form $\qform:V\rightarrow\RR$ is defined by the equation 
\[ 
\qform(\bx)= \langle \bx,\bx\rangle = x^2+y^2+z^2-2(xy+xz+yz). 
\]
The \dfn{norm} of $\bx$ is $\qform(\bx)$.

Since $O(2,1)$ preserves the bilinear form, it acts on the level sets of $\qform$.
There are three level sets of particular interest in our paper.
The level set defined by the equation $\qform(\bx)=0$ is a double cone through the origin called the \dfn{isotropic cone}. Vectors in the isotropic cone are called \dfn{isotropic vectors}.
The level set defined by $\qform(\bx)=1$ is a hyperboloid of one sheet, and the level set defined by $\qform(\bx)=-1$ is a hyperboloid of two sheets.
The latter two level sets are separated by the isotropic cone.

A \dfn{reflection} is a linear transformation of order $2$ that fixes a hyperplane pointwise.
For a nonisotropic vector $\bv$, we define the reflection $r_{\bv}\colon V\to V$ by 
\[
r_{\bv}(\bx) = \bx - 2\frac{\langle \bv,\bx\rangle}{\langle \bv,\bv\rangle} \bv.
\]
This reflection $r_{\bv}$ is an orthogonal transformation that fixes the plane $H_{\bv}=\{ \bx:\langle \bv,\bx\rangle=0\}$ pointwise and swaps the two half-spaces $\cH_{\bv}^+=\{ \bx:\langle \bv,\bx\rangle>0\}$ and $\cH_{\bv}^- = \{\bx: \langle \bv,\bx\rangle<0\}$.
If $\qform(\bv)=1$, the reflection takes the simpler form $r_{\bv}(\bx) = \bx - 2\langle \bv,\bx\rangle \bv$.
From now on, we only consider reflections $r_{\bv}$ where $\bv$ has norm $1$.

Next we record a few easy facts we will need.

\begin{proposition}\label{prop:ortho_planes}
    Let $\bu\in V$ be a nonzero vector.
    \begin{enumerate}
        \item If $\qform(\bu)=0$, then $H_{\bu}$ is the plane containing $
        \bu$ tangent to the isotropic cone.
        \item If $\qform(\bu)>0$, then $H_{\bu}=\Span\{\bv,\bv'\}$ for some pair of isotropic vectors $\bv,\bv'$.
        \item If $\qform(\bu)<0$, then the nonzero vectors in $H_{\bu}$ lie outside the isotropic cone.
    \end{enumerate}
\end{proposition}

%
%

Let $X^\perp$ denote the orthogonal complement of a linear subspace $X$ of $V$. 

\begin{lemma}\label{lem:orthogonal_arrangement}
    Let $L$ be a line through the origin, and let $\cA$ be an arrangement of planes through the origin.
    Then the line $L$ is contained in each plane in $\cA$ if and only if each line in the set ${\cA^{\perp}=\{H^{\perp}:H\in\cA\}}$ is contained in the plane $L^{\perp}$.
\end{lemma}

\begin{proof}
Orthogonal complementation reverses inclusion of subspaces.
Applying this fact to a collection of subspaces proves the result.
\end{proof}

Let $\cL$ be the set of lines through the origin.
This set decomposes into three orbits under the $O(2,1)$ action.
Namely, the orbits are $\cL_{<0},\ \cL_0$, and $\cL_{>0}$, which consist of the lines spanned by a vector of negative norm, zero norm, and positive norm, respectively.

Consider the affine plane \[\AA = \{(x,y,z):x+y+z=1\}\] containing $\Delta$, and let $\AA_0 = \{(x,y,z):x+y+z=0\}$ be the linear plane parallel to $\AA$.
The convex hull of $\Delta$, denoted $\conv(\Delta)$, is an equilateral triangle in $\AA$.
Let \[\DD = \{\bx\in\AA: \qform(\bx)\leq 0\}.\]

\begin{lemma}\label{lem:disk}
    The plane spanned by any two elements of $\Delta$ is tangent to the isotropic cone.
\end{lemma}

\begin{proof}
    By symmetry, it is enough to prove the statement for the plane $H$ spanned by $\alpha_1,\alpha_2$.
    We compute
    \[ 
    \qform(x\alpha_1+y\alpha_2) = x^2+y^2-2xy = (x-y)^2.
    \]
    Hence, the intersection of $H$ with the isotropic cone is the line spanned by $\alpha_1+\alpha_2$.
    We conclude that $H$ is tangent to the isotropic cone at that line.
\end{proof}

By intersecting with the plane $\AA$, we deduce from \cref{lem:disk} that the boundary of $\DD$ is the inscribed circle of the triangle $\conv(\Delta)$. 
The center of $\DD$ is the point \[O=\left(\frac{1}{3},\frac{1}{3},\frac{1}{3}\right).\]
Each line in $\cL_{<0}$ contains a unique point in the interior $\DD^{\circ}$, and each line in $\cL_0$ contains a unique point on the boundary $\partial\DD$.

\begin{remark}\label{rem:proj_action}
    The linear action of $O(2,1)$ on $V$ induces a projective linear action of $O(2,1)$ on $\DD$.
\end{remark}

For $\bx\in V\setminus\AA_0$, let $\wh{\bx}\in\AA$ be the unique point such that $\bx$ is on the line through $\wh{\bx}$ and the origin. For $\mathscr X\subseteq V\setminus\AA_0$, let $\widehat {\mathscr X}=\{\wh{\bf x}:\bf x\in\mathscr X\}$. In particular, if $W$ is a linear subspace of $V$ not contained in $\AA_0$, then $\wh{W}=W\cap \AA$.
We may translate \cref{prop:ortho_planes} to the affine plane $\AA$ as follows.

\begin{proposition}\label{prop:affine_orthogonal_subspaces}
    Let $u \in\AA$.
    \begin{enumerate}
        \item If $u\in\partial\DD$, then $\wh{H}_u$ is the tangent line to $\partial\DD$ at $u$.
        \item If $u\in\AA\setminus\DD$, then $\wh{H}_u$ is the line in $\AA$ through the two distinct points $p,q\in\partial \DD$ such that $u\in \wh{H}_{p}\cap \wh{H}_{q}$.
        \item If $u\in\DD^{\circ}$ and $u\neq O$, then $\wh{H}_u\cap\DD=\varnothing$.
    \end{enumerate}
\end{proposition}

For $\bx$ of positive norm, $\bx$ is in the half-space $\cH_{\bx}^+$ by definition.
Combined with the observation that $H_O = \AA_0$ and $O\in\cH_O^-$, we deduce the following corollary of \cref{prop:affine_orthogonal_subspaces}.

\begin{corollary}\label{cor:circle_decomposition}
    Given $u\in\AA\setminus\DD$, let $p,q$ be the two distinct points in $\wh{H}_u\cap\partial\DD$.
    The line $\wh{H}_u$ separates $\DD$ into two regions.
    The set $\wh{\cH}_u^-\cap \DD$ is the region containing $O$. 
    The other region, $\wh{\cH}_u^+\cap \DD$, is contained in the triangle $\conv(u,p,q)$.
\end{corollary}

We illustrate \cref{prop:affine_orthogonal_subspaces} and \cref{cor:circle_decomposition} in \cref{fig:orthogonal}. Consider distinct points $p,q$ on the boundary of $\DD$.
The lines $\wh{H}_p$ and $\wh{H}_q$ are tangent to the circle and meet at a point $u$ outside the disk.
Then $\wh{H}_u$ is the line through $p$ and $q$.
If $v$ is the midpoint of the line segment between $p$ and $q$, then $\wh{H}_v$ is the line through $u$ parallel to $\wh{H}_u$.

\begin{figure}
  \includegraphics[height=75.658mm]{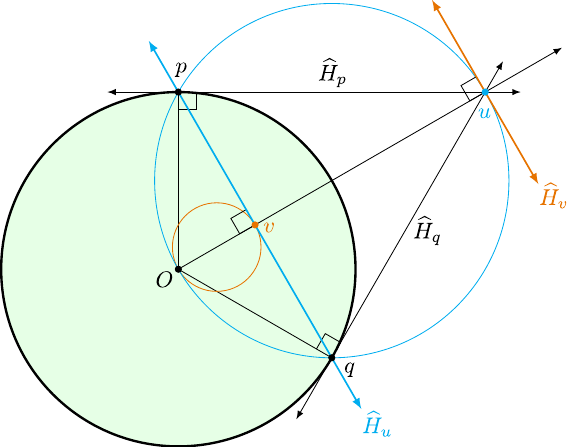}
  \caption{The points in $\AA$ other than $O$ are in bijection with the
    lines in $\AA$ not through $O$ using closest point projection and
    inversion.  Under complementation, the point $v \in \AA$
    corresponds to the line $\widehat H_v$, and the point $u$ corresponds to
    the line $\widehat H_u$.\label{fig:orthogonal}}
\end{figure}

We define angles between lines in the affine plane $\AA$ by projecting them to the linear plane $\AA_0$, which is a subspace of $V$ on which $\langle -,-\rangle$ is positive definite.
The following lemma is immediate.

\begin{lemma}\label{lem:right_angle}
    Suppose $H$ is a plane such that $H\neq \AA_0$ and $O\notin H$. 
    If $L=H^{\perp}$ is the line orthogonal to $H$, then the ray from $O$ towards $\wh{L}$ intersects the line $\wh{H}$ at a right angle.
\end{lemma}

By taking an intersection with the affine plane $\AA$, we translate \cref{lem:orthogonal_arrangement} into the collinearity statement of \cref{lem:collinear_order}. 

\begin{lemma}\label{lem:collinear_order}
    Let $u\in\AA\setminus\{O\}$, and let $\cA$ be an arrangement of linear planes that each contain $u$ but not $O$.
    Then $\wh{\cA^{\perp}} = \{\wh{H^{\perp}}:H\in\cA\}$ is a collinear set of points in $\AA$.
    Moreover, if the lines $\wh{H}$ for $H\in\cA$ are cyclically ordered around $u$ starting from $O$, then the points $\wh{H^{\perp}}$ are ordered in the same way along the line $\wh{u^{\perp}}$. 
\end{lemma}

\begin{proof}
    The collinearity is immediate from \cref{lem:orthogonal_arrangement} by restricting to the affine plane $\AA$.
    For the second statement, it is enough to consider three planes at a time.
    
    Suppose $\cA=\{H_1,H_2,H_3\}$, where $H_2$ is not incident to the region of $\AA\setminus\bigcup\cA$ containing $O$.
    For each $i$, consider the ray $\rho_i$ starting at $O$ that is perpendicular to $\wh{H}_i$.
    By \cref{lem:right_angle}, the point $u_i = \wh{H_i^{\perp}}$ is on the ray $\rho_i$.
    Since $\rho_2$ is between $\rho_1$ and $\rho_3$, the point $u_2$ is between $u_1$ and $u_3$; see \cref{fig:collinear_fig}.
\end{proof}


\begin{figure}
    \includegraphics[height=80.370mm]{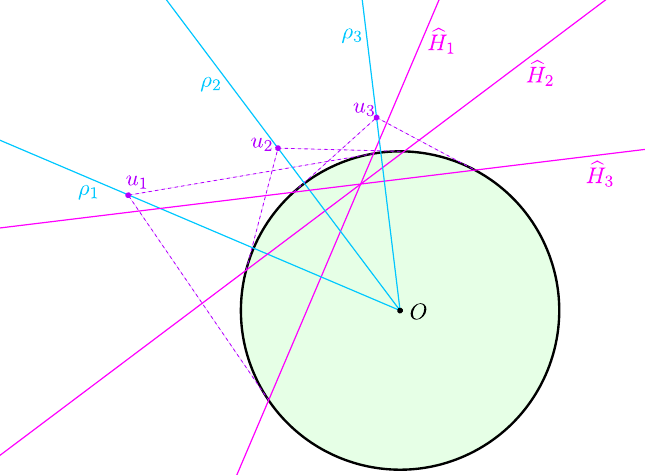}
    \caption{An illustration of \cref{lem:collinear_order} \label{fig:collinear_fig}}
\end{figure}

\begin{remark}\label{rem:klein}
  For any line $\ell$ in $\AA$ that passes through the interior of $\DD$, the line segment $\ell\cap\DD$ is a geodesic between ideal points in the Klein model of the hyperbolic plane.
  The projective linear action of a reflection $r_{\bf v}$ on $\DD$ is an isometry of the Klein model.
  For the purpose of visualization, we occasionally transform the Klein model into the Poincar\'e model. See \cite{CFP97} 
  for the relations between the different models of hyperbolic space.
  Geodesics in the Poincar\'e model are arcs of circles orthogonal to the boundary of $\DD$.
  However, all of our computations will take place in the Klein model, so we do not give details of the Poincar\'e model. 
\end{remark}

\section{The Universal Coxeter Group}\label{sec:universal}

This section recalls the definition of the group $U_3$, its standard
representation as a reflection group, and a realization of its Coxeter complex in the disk.  See  
\cite{Humphreys, bjorner:2005Coxbook} for further background on root systems 
and \cite{DyerHohlwegRipoll} for how to represent roots in an affine patch.

\begin{definition}[Universal Coxeter Group]\label{def:universal}
  The \dfn{universal Coxeter group of rank~$n$} is the group $U_n$ with presentation 
  \[U_n=\langle s_1,\ldots,s_n\mid s_1^2=\cdots=s_n^2=\id\rangle,\] where we write $\id$ for the identity element. In other words, $U_n$ is the free product of $n$ copies of the cyclic group of order $2$. 
\end{definition}

The group $U_1$ is the cyclic group of order $2$, and the group $U_2$ is the infinite dihedral group. 
 Since the generators
  are involutions, elements of $U_3$ are described by positive words
  over the alphabet $S$.  In fact, elements are in bijection with
  positive \dfn{reduced} words, in which adjacent letters are distinct.

\subsection{The standard reflection representation}

 Every Coxeter group with $n$ Coxeter generators has a faithful representation on a vector space of dimension
  $n$ preserving a symmetric bilinear form.  We now describe one such representation for $U_3$. 

Let $V=\RR^3$ be the vector space with distinguished basis $\Delta=\{\alpha_1,\alpha_2,\alpha_3\}$ and bilinear form $\langle -,-\rangle$ defined in \cref{sec:hyperbolic}.

\begin{definition}
The map $s_i\mapsto r_{\alpha_i}$ extends to a faithful representation of $U_3$ on $V$ preserving the form $\langle -,-\rangle$. 
This representation is the \dfn{standard reflection representation} of $U_3$.
\end{definition}

By \cref{rem:proj_action}, we obtain a projective linear representation of $U_3$ on $\DD$ as follows.

\begin{proposition}
    The linear action of $U_3$ on $V$ induces a projective linear action $\phi$ on $\DD$ such that for $w\in U_3$ and $\bx\in\DD$, we have $\phi(w)(\bx) = \wh{w\cdot \bx}$.
\end{proposition}

\begin{definition}[Roots]\label{def:roots}
A \dfn{root} is any vector of the form $w\cdot \alpha_i$ for some $w\in U_3,\ i\in\{1,2,3\}$.
The set of roots $\Phi$ is the \dfn{root system}.
Elements of $\Delta$ are \dfn{simple roots}.
\end{definition}

\begin{figure}
  \includegraphics[width=11cm]{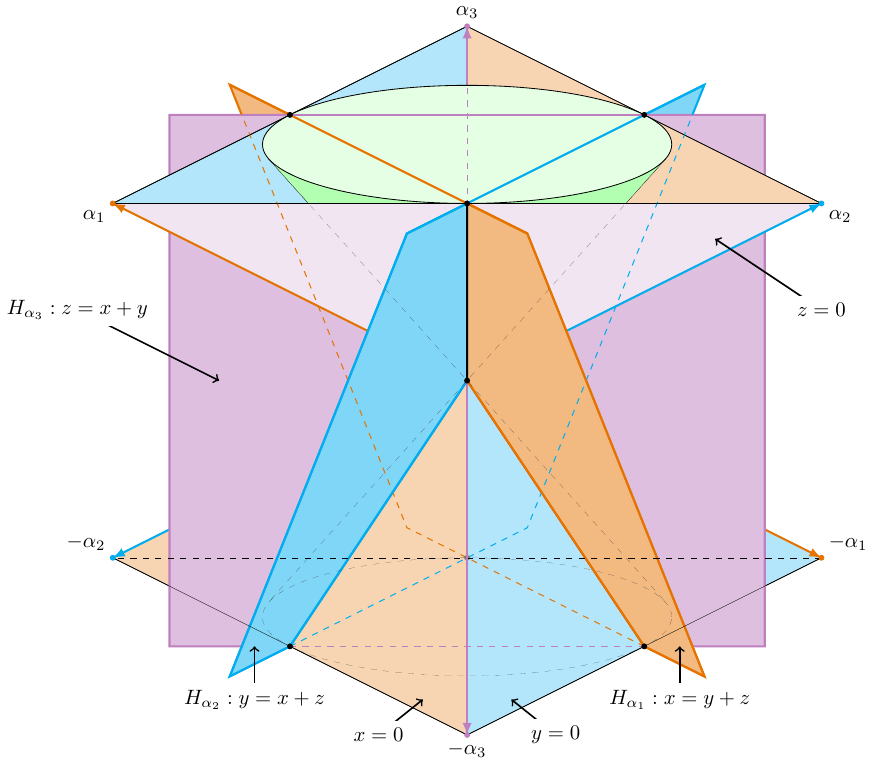}
  \caption{The action of $U_3$ on $V = \RR^3$ preserves the form $\langle - , - \rangle$.
    Under this action, the double-cone of vectors of norm $0$ is
    preserved setwise.  The generator $s_i$ sends the root $\alpha_i$
    to its negative and pointwise fixes the hyperplane
    $H_{\alpha_i}$.  The coordinate hyperplanes are included for
    reference.
    \label{fig:hyperboloid}}
\end{figure}

\begin{definition}[Coordinates]\label{def:coordinates} 
  Using the coordinate system for $V$ defined by $\Delta$,
  we can identify the generators $s_i$ with
  the square matrices that act on $V$ by left multiplication on column
  vectors.  Concretely,
  \[
  s_1 = \left[ \begin{array}{rrr} -1 & 2 & 2 \\ 0 & 1 & 0\\ 0 & 0 &
      1\end{array} \right],\
\quad   
  s_2 = \left[ \begin{array}{rrr} 1 & 0 & 0 \\ 2 & -1 & 2\\ 0 & 0 &
      1\end{array} \right],\ \quad
  s_3 = \left[ \begin{array}{rrr} 1 & 0 & 0 \\ 0 & 1 & 0\\ 2 & 2 &
      -1\end{array} \right].
  \]
  The generator $s_i$ pointwise fixes a hyperplane $H_{\alpha_i}$. Concretely, $H_{\alpha_1}$ is defined
  by the equation $x=y+z$, $H_{\alpha_2}$ is defined by
  the equation $y=x+z$, and $H_{\alpha_3}$ is defined by the equation
  $z=x+y$.  See \cref{fig:hyperboloid}. 
\end{definition}

\begin{figure}
  \includegraphics[width=11.7cm]{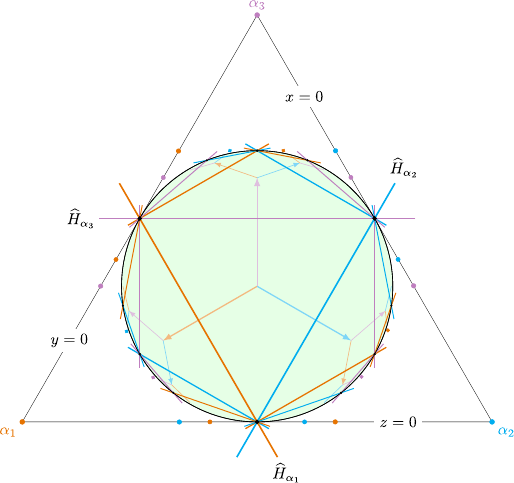}
  \caption{When \cref{fig:hyperboloid} is restricted to the
    affine plane $\AA$ with equation $x+y+z=1$, it contains the Klein
    model of the hyperbolic plane. The action of $U_3$ on (the lines
    through) this model preserves both the Farey tessellation by ideal
    triangles and its dual trivalent tree. The colored dots outside of $\DD$ depict some of the rescaled roots. 
    \label{fig:klein}}
\end{figure}

\begin{definition}[Positive and Negative Roots]\label{def:pos-roots}
  The \dfn{nonnegative span} of
  a set $\mathscr X\subseteq V=\mathbb R^3$ is the collection $\spange(\mathscr X)$ of all nonnegative linear
  combinations of vectors in $\mathscr X$. The roots in $\Phi^+ = \Phi \cap \spange(\Delta)$ are \dfn{positive},
  and the roots in $\Phi^- = \Phi \cap \spange(-\Delta)=-\Phi^+$ are
  \dfn{negative}.  
\end{definition}

Since every root is positive or negative, we have $\Phi =
  \Phi^+ \sqcup \Phi^-$. Note that the nonnegative span of $\Delta =
  \{\alpha_1,\alpha_2,\alpha_3\}$ is an octant in $\mathbb R^3$.

\begin{definition}[Linear Subsystems]\label{def:linear-subs}
  A \dfn{rank~$2$ linear subsystem} is a subset of $\Phi$ of the
  form $H\cap \Phi$, where $H$ is a $2$-dimensional subspace in $V$ spanned by $H\cap \Phi$.
  A linear subsystem $\Phi'$ can be partitioned into positive and negative roots by the hyperplane $\AA_0$.
  The extreme vectors of $\Phi'^+$ are said to be \dfn{relatively simple} in $\Phi'$.
\end{definition}

\begin{definition}[Parabolic Subsystems]
  Important special cases of linear subsystems are the \dfn{standard rank~$2$
    parabolic subsystem} $\Phi_{ij} = \Span
  \{\alpha_i,\alpha_j\}\cap \Phi$ and the \dfn{rank~$2$ parabolic
    subsystem} $w \cdot \Phi_{ij} = \Span
  \{w(\alpha_i),w(\alpha_j)\} \cap \Phi$, where $\alpha_i,\alpha_j$
  are distinct simple roots and $w\in U_3$.
\end{definition}

The elements of $\wh \Phi^+ = \{\wh \beta \mid \beta \in \Phi^+\}$ are called \dfn{rescaled (positive) roots}.
Since roots have positive norms, $\wh{\Phi}^+$ is outside the disk.
In other words, $\wh \Phi^+ \subset \conv(\Delta)\setminus \DD$.

\subsection{The Farey Tessellation}\label{subsec:farey}

In this section, we recall the Coxeter complex and present a geometric realization of the complex as a triangulation of $\DD$. The realization below is the specialization to $U_3$ of the projective root system and imaginary cone picture developed in \cite{HohlwegLabbeRipoll,DyerHohlwegRipoll}. In our normalization, the affine disk $\DD$ is the Klein model determined by the negative cone of the Lorentzian form, while $\partial\DD$ is the projectivized isotropic cone. The tessellation by ideal triangles used here is the Coxeter tessellation in this projective model. 

 A \dfn{standard parabolic subgroup} of $U_3$ is a subgroup $U_J$ generated by a proper subset $J \subset S$.  
 A \dfn{parabolic subgroup} is a conjugate $w U_J w^{-1}$ of a standard parabolic subgroup.
  A \dfn{parabolic coset} is a left coset $w U_J$ of a
 standard parabolic subgroup.  The \dfn{rank} of either is the
 size of $J$. The parabolic cosets of $U_3$ under reverse inclusion form the
 face lattice of a simplicial complex called the \dfn{Coxeter
   complex} of $U_3$.

\begin{remark}
    A parabolic subgroup $w U_J w^{-1}$ is a Coxeter group with generating set $wJw^{-1}$. 
    The parabolic subsystem $w\cdot \Phi_J$ is the root system for the restriction of the linear representation to $wU_Jw^{-1}$.
\end{remark}

Note that the parabolic subgroup $w U_J
 w^{-1}$ is the left stabilizer of the parabolic coset $w
 U_J$. We will simplify notation by writing, e.g., $U_{s_1,s_2}$ and $U_{s_1}$ instead of $U_{\{s_1,s_2\}}$ and $U_{\{s_1\}}$. The following observation will be useful later in the paper.  

\begin{remark}[Parabolic Cosets in $U_3$]\label{rem:cosets}
  The seven standard parabolic subgroups of $U_3$ are $U_{s_1,s_2}$,
  $U_{s_2,s_3}$, $U_{s_1,s_3}$, $U_{s_1}$, $U_{s_2}$, $U_{s_3}$, and $U_{\varnothing}$.
  The first three are rank $2$ infinite
  dihedral groups isomorphic to $U_2$.  The next three are rank~$1$ subgroups of order $2$ isomorphic to $U_1$.
  The last one is the trivial group $U_0$ with rank~$0$.  Note that
  for every reduced word $w$ and proper subset $J \subset S$, there is
  a unique factorization $w = w_1 w_2$ such that the suffix $w_2$ is a
  word over $J$ and there is no longer suffix with this
  property.  Since $w U_J = w_1
  U_J$, we may always assume that $w$ either is trivial or
  ends in a letter not in $J$.  This shortest representative of the
  coset is unique.
\end{remark}

The next proposition is standard, but we include a short proof for the sake of completeness. 

\begin{proposition}\label{prop:coset_subgroup_bij}
    Each nontrivial parabolic subgroup of $U_3$ stabilizes a unique parabolic coset.
\end{proposition}

\begin{proof}
    Suppose $P$ is a parabolic subgroup that stabilizes $wU_J$ and $vU_{J'}$.
    This means the subgroups $wU_Jw^{-1}$ and $vU_{J'}v^{-1}$ are equal.
    Hence, $U_{J'} = uU_Ju^{-1}$, where $u=v^{-1}w$.
    Since $J'\neq \varnothing$, there is a simple generator in $uU_Ju^{-1}$.
    Since $U_3$ is the universal Coxeter group, this is impossible unless every generator in a reduced word for $u$ is in $U_J$ (i.e., $u\in U_J$).
    Therefore, $J=J'$ and $wU_J=vU_J$. 
\end{proof}

\begin{definition}[Farey Tessellation $\cT$]\label{def:farey}
  Following \cref{rem:klein}, we refer to points on $\partial \DD$ as \dfn{ideal points}.
  An \dfn{ideal triangle} is a triangle whose vertices are ideal points.
  A fundamental domain for the action of $U_3$ on $\DD$ is the central
  ideal triangle $T_{\id}=\{v\in\DD:\forall \beta\in\Phi^+, \langle\beta,v\rangle\leq 0\}$ in \cref{fig:klein} bounded by the lines
  $\wh H_{\alpha_i},\ i\in\{1,2,3\}$.  
  Its orbit under the
  $U_3$ action covers the interior of the disk $\DD$, and the orbits of its
  three ideal vertices form a countable collection of points in
  $\partial \DD$.  This is the \dfn{Farey tessellation $\cT$}.  There
  is a bijection $w\mapsto T_w$ from $U_3$ to the set of triangles in
  $\cT$, where $T_w=w \cdot T_{\id}$.  The left-hand side of \cref{fig:farey-cayley} shows the
  Farey tessellation redrawn in the Poincar\'e disk model, where its
  structure is easier to see.  The triangles near the center are
  labeled.
\end{definition}

\begin{figure}
  \begin{tabular}{ccc}
    $\begin{array}{c}\includegraphics[height=70mm]{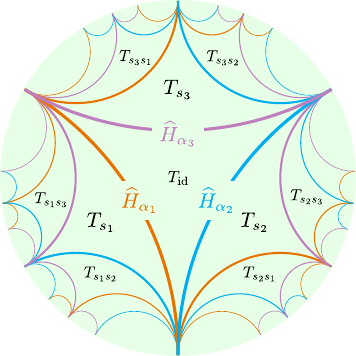}\end{array}$
    & \hspace*{2mm} &
    $\begin{array}{c}\includegraphics[height=70mm]{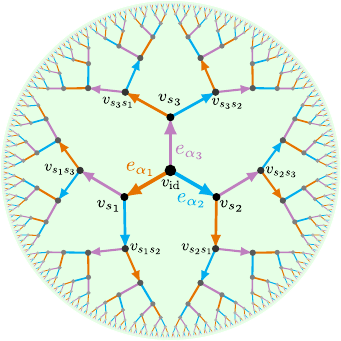}\end{array}$
  \end{tabular}
  \caption{The image on the left is the Farey tessellation $\cT$ drawn
    in the Poincar\'e disk model. The image on the right is the
    trivalent tree Cayley graph $\Gamma$ dual to the tessellation.
    Triangles, hyperplanes, and ideal points on the left correspond to
    vertices, edges, and regions on the right. 
    \label{fig:farey-cayley}} 
\end{figure}

\begin{remark}[Stabilizers]\label{rem:stabilizers}
  The stabilizers of the sides and vertices of the triangle $T_\id$ are
  the standard parabolic subgroups of $U_3$ (\cref{rem:cosets}).
  More generally, the hyperplanes in $\cT$ are in bijection with the
  rank~$1$ parabolic cosets $w U_{s_i}$, and the ideal points in
  $\cT$ are in bijection with the rank $2$ parabolic cosets $w
  U_{s_i,s_j}$.  
\end{remark}

\begin{remark}
    For a general Coxeter group $W$, there is an analogous cellular decomposition of a cone known as the Tits cone.
    Each face in this decomposition corresponds to a parabolic coset for $W$. 
    The stabilizer of a face is the parabolic subgroup that stabilizes the corresponding parabolic coset.
    Two faces have the same stabilizer if and only if they have the same linear span. 
    For the universal Coxeter group, the nontrivial parabolic cosets and subgroups are in natural bijection by \cref{prop:coset_subgroup_bij}.
    Thus, we may identify vertices and edges of $\cT$ by their linear span.
    See, for example, \cite[Section II.5.13]{Humphreys}.
\end{remark}

Let $\Gamma=\Cay(U_3,S)$ be the right Cayley graph of $U_3$ generated by $S$. The vertices of $\Gamma$ are elements of $U_3$; often when we wish to stress that we are viewing an element $w\in W$ as a vertex of $\Gamma$, we will denote it by $v_w$. The graph $\Gamma$
  naturally embeds in $\DD$ as the dual graph of the Farey tessellation $\cT$.  
  For this embedding, we let $v_w=w\cdot O \in T_w$.
  The hyperplane
  between $T_w$ and $T_{ws_i}$ is dual to the edge with endpoints
  $v_w$ and $v_{ws_i}$.  We denote this edge by $e_\beta$, where $\beta\in\Phi^+$ is the unique positive root such that $r_\beta=ws_iw^{-1}$. Since the triangles are ideal, the Cayley graph $\Gamma$ is a
  tree with unique shortest paths.  We orient the edges of $\Gamma$ away from the
  identity vertex~$v_\id$.  The right-hand side of
  \cref{fig:farey-cayley} shows a portion of the Cayley graph.
  The image is schematic so that more details can be shown.  The
  connected components of $\DD \setminus \Gamma$, which we call \dfn{faces}, are in natural
  bijection with the ideal vertices of the tessellation $\cT$ and also
  the rank~$2$ parabolic cosets. See \cref{tbl:correspondences}.

\begin{table}
  \begin{center}
    {\renewcommand{\arraystretch}{1.3}
    \begin{tabular}{|c||c|c|c|c|}
      \hline Rank & Group $U_3$ & Tiling $\mathcal{T}$ & Graph
      $\Gamma$\\ \hline
      
      \hline rank $0$ & $w\cdot U_{\varnothing}$ & triangle $T_w$ &
      vertex $v_w$\\
      
      \hline rank $1$ & $w\cdot U_{s_i}$ & hyperplane $\wh
      H_\beta$ & edge $e_\beta$\\ \hline
    \end{tabular}
    }
  \end{center}
  \caption{The parabolic cosets of rank $0$ and $1$ in $U_3$ label
    the triangles and hyperplanes of the Farey tessellation $\cT$ and
    the vertices and edges of the Cayley graph $\Gamma$.  The positive root $\beta \in \Phi^+$ corresponds to the reflection $
    ws_iw^{-1}$.\label{tbl:correspondences}}
\end{table}

\begin{definition}[Orientation of the Cayley graph]
Since the Cayley graph $\Gamma$ is a tree, it can be oriented away from the vertex $v_\id$. Thus, $v_\id$ has outdegree $3$, and every other vertex has outdegree $2$. 
\end{definition}

\begin{proposition} \label{new-parabolic-prop}
There are natural bijections between
\begin{enumerate}
    \item the faces of the Cayley graph (as embedded in $\DD$);
    \item the vertices of the Farey tessellation;
    \item the rank 2 parabolic root subsystems; 
    \item the rank 2 parabolic cosets.
\end{enumerate}
    
Specifically, given a face $f$ of the Cayley graph, there is a rank 2 parabolic coset $vU_{s_i,s_j}$ such that the vertices on the boundary of $f$ are $v_w$ for $w \in v U_{s_i,s_j}$, the edges on the boundary of $f$ are $e_{\beta}$ for $\beta \in v \cdot \Phi_{ij}$, and the vertex of the Farey tessellation is the common point in $\{\wh{H}_{\beta}:\beta\in v\cdot \Phi_{ij}\}$. 
\end{proposition}

We now discuss in more detail the correspondence between the edges of $\partial f$ and the rank 2 parabolic subsystem above.
In the face $f$, there is a unique \dfn{source vertex} $u_0$ such that $\partial f$ is oriented as $\cdots u_{-2} \leftarrow u_{-1} \leftarrow u_0 \rightarrow u_1 \rightarrow u_2 \rightarrow \cdots$.
We define $\yy(f)$ to be the group element such that $u_0 = v_{\yy(f)}$.
The following statement follows from \cref{lem:collinear_order}.

\begin{proposition} \label{parabolic-geometry}
    With notation as above, the vertices of $f$ are $\{ v_w : w \in \yy(f) U_{s_i,s_j} \}$ and the edges of $f$ are $\{ e_{\beta} : \beta \in \yy(f) \Phi_{ij} \}$. For $j>0$, let $e_{\beta_j}$ be the edge $u_{j-1} \to u_j$ and, for $j<0$, let $e_{\beta_j}$ be the edge $u_{j+1} \to u_j$. The rescaled roots $\{ \wh\beta : \beta \in \yy(f) \Phi_{ij} \}$ lie on a line segment in $\AA$, in the order $\wh\beta_{-1}$, $\wh\beta_{-2}$, $\wh\beta_{-3}$, \dots, $\wh\beta_3$, $\wh\beta_2$, $\wh\beta_1$. In particular, the relatively simple roots are $\beta_{-1}$ and $\beta_1$, and $\{ \beta_{-1}, \beta_1 \} = \yy(f) \{ \alpha_i, \alpha_j \}$. 
\end{proposition}

\begin{proposition}\label{prop:rank2_parabolics}
    Every positive root $\beta$ is in precisely two rank 2 parabolic subsystems. 
    If $\beta\notin\Delta$, then $\beta$ is relatively simple in exactly one of the rank 2 parabolic subsystems that contain it.
\end{proposition}

\begin{proof}
    The faces of the Cayley graph $\Gamma$ correspond to the rank 2 parabolic subsystems of $\Phi$ by \cref{new-parabolic-prop}.
    Since each positive root $\beta$ labels exactly one edge of the Cayley graph, the two faces containing $e_{\beta}$ correspond to the only rank 2 parabolic subsystems that contain $\beta$.
    This proves the first claim.
    
    For the second claim, let $\beta$ be a positive root that is not simple. Let $u$ and $v$ be the endpoints of $e_\beta$, with $e_\beta$ directed $u \to v$. Since $\beta$ is not a simple root, $u$ is not the identity, so $u$ has indegree $1$. Let the other edges incident to $u$ be $x \to u$ and $u \to y$. Let $f_1$ be the face between $x \to u$ and $u \to v$, and let $f_2$ be the face between $u \to v$ and $u \to y$. Then $v$ is the source vertex of $f_2$, but not of $f_1$, so, by \cref{new-parabolic-prop}, $\beta$ is relatively simple in the subsystem corresponding to $f_2$ but not~$f_1$. 
\end{proof}

\begin{remark}
In general Coxeter groups, the positive roots correspond to
hyperplanes, and the rank~$1$ parabolic cosets correspond to edges in
the Cayley graph, but the map from the latter to the former is
many-to-one since hyperplanes cross many edges.  For $U_3$, however,
each hyperplane crosses a unique edge, and all four sets are in
one-to-one correspondence, analogous to \cref{new-parabolic-prop}.
\end{remark}

For inductive proofs, we use the following statistics on positive roots. Note that the typical definition of depth for arbitrary Coxeter groups found in \cite{bjorner:2005Coxbook} is phrased differently from our definition. However, the definition we give is equivalent for $U_3$. The definition we present here is better suited given the setup we have provided so far. 

\begin{definition}[Depth and Height]
The \dfn{depth} of a positive root $\beta$ is the graph-theoretic distance in $\Gamma$ between $v_\id$ and the endpoint of $e_\beta$ that is farther from $v_\id$. We similarly define the \dfn{depth} of the edge $e_\beta$ to be the depth of $\beta$. Because $\Delta$ is a basis of $V$, there are unique nonnegative real numbers $x,y,z$ such that $\beta=x\alpha_1+y\alpha_2+z\alpha_3$. The \dfn{height} of $\beta$, denoted $\height{\beta}$, is $x+y+z$. 
\end{definition}

\section{Inversion Sets and Notions of Convexity}\label{sec:invert-convex}

This section recalls Dyer's notion of extended
weak order from \cite{dyer:2019weak}, along with related concepts.  
We begin with inversion sets and weak order.

\begin{definition}[Inversion Sets and Weak Order]\label{def:weak-order}
  A positive root $\beta\in\Phi^+$ is an \dfn{inversion} of an element
  $w\in U_3$ if $\beta \in w\cdot \Phi^-$ or, equivalently, if
  $w^{-1}\cdot \beta\in\Phi^-$.  
  Let $\Inv(w)$ denote the set of
  inversions of $w$.  The \dfn{weak order} is the partial order on $U_3$ defined by containment of inversion
  sets: for $w, w'\in U_3$, we have $w \leq w'$ if and only if
  $\Inv(w)\subseteq\Inv(w')$.
\end{definition}

\begin{remark}[Inversions and Tree Paths]\label{rem:geodesics} 
  A root $\beta\in\Phi^+$ is in $\Inv(w)$ if and only if the edge $e_\beta$
  lies in the unique non-backtracking path through the Cayley graph from $v_\id$ to $v_w$, and this occurs if and only if
  the line $\wh H_\beta$ separates
  the center of $T_\id$ from the center of $T_w$.  As a consequence,
  the oriented Cayley graph $\Gamma$ is the Hasse diagram of the weak order. See \cref{fig:farey-cayley}. 
\end{remark}

\begin{remark}[Root Subsets and Colored Edges]\label{rem:colors}
  We represent a bipartition $R \sqcup B = \Phi^+$ (and the corresponding bipartition $\wh R\sqcup\wh B=\wh\Phi^+$) by bicoloring the corresponding edges of $\Cay(U_3,S)$. Our convention is that edges corresponding to roots in $R$ are
  \textcolor{red}{red}, while edges corresponding to roots in $B$ are \textcolor{blue}{blue}.  See
  \cref{fig:finite}.
\end{remark}

\begin{definition}[Separable]
  A subset $R \subseteq \Phi^+$ is \dfn{separable} if there is a
  linear hyperplane $H$ in $V$ with $H \cap \Phi = \varnothing$ such that $R$ and $\Phi^+\setminus R$ lie on opposite sides of~$H$.  
\end{definition} 

\begin{remark}
For each $w\in U_3$, the inversion set $\Inv(w)$ is separable. 
Namely, the hyperplane $w\cdot \AA_0$ separates $\Inv(w)$ and $\Phi^+\setminus \Inv(w)$.
\end{remark}

The following definition appears in the work of Hohlweg and Labb\'e \cite{HohlwegLabbe2016} (where it is called \emph{separable}). 

\begin{definition}[Weakly Separable]\label{def:weakly-sep}
Let $R \subseteq \Phi^+$, and let $B=\Phi^+\setminus R$. We say $R$ is \dfn{weakly separable}
  if $\spange R\cap\spange B=\{0\}$. Equivalently, $R$ is weakly separable if the convex hulls of $\wh R$ and $\wh B$ are disjoint. 
\end{definition}

Note that separable sets are weakly
  separable.

\begin{definition}[Biclosed Sets and Extended Weak Order]\label{def:biclosed}
  A set $R\subseteq\Phi^+$ is \dfn{closed} if for all $\alpha,\beta\in R$, the set $\spange \{\alpha,\beta\}\cap \Phi^+$ is contained in $R$. We say $R$ is \dfn{biclosed} if $R$ and $\Phi^+\setminus R$ are both closed. The 
  \dfn{extended weak order} of $U_3$ is the collection of biclosed subsets of $\Phi^+$, partially
  ordered by containment.  
\end{definition} 
Since the convexity condition defining
  weakly separable is stronger than that defining biclosed, all
  weakly separable sets are biclosed. To summarize, inversion sets and their complements are separable,
separable sets are weakly separable, and weakly separable sets are
biclosed.  In the other direction, one can show that the only finite
biclosed sets are inversion sets \cite[Lemma 4.1]{dyer:2019weak}, which shows that the
weak order embeds as an order ideal of the extended weak order. 

Let $\Phi'\subseteq \Phi$ be a rank 2 linear subsystem. A set $R'\subseteq\Phi'$ is \dfn{closed in $\Phi'^+$} if for all $\alpha,\beta\in R'$, the set $\spange\{\alpha,\beta\}\cap\Phi'^+$ is contained in $R'$. The set $R'$ is \dfn{biclosed in $\Phi'^+$} if both $R'$ and $\Phi'^+\setminus R'$ are closed in $\Phi'^+$. 
    
\begin{definition}[Parabolic Biclosed Sets]\label{def:parabolic_biclosed}
    A set $R\subseteq\Phi^+$ is \dfn{parabolic biclosed} if $R\cap \Phi'$ is biclosed in $\Phi'^+$ for every rank 2 parabolic subsystem $\Phi'\subseteq \Phi$.  
\end{definition} 

Every biclosed set is parabolic biclosed because intersecting a globally closed and coclosed subset of $\Phi^+$ with a rank 2 parabolic subsystem preserves closedness and coclosedness inside that subsystem. The converse need not hold; parabolic biclosedness only tests the rank 2 parabolic subsystems, whereas global biclosedness also sees rank 2 linear subsystems that are not parabolic. \cref{fig:parabolicnotbiclosed} gives examples of parabolic biclosed sets with the same associated pair of snakes, only one of which is globally biclosed. We introduce this weaker condition because the construction of snakes in \cref{sec:parabolic} uses only color changes along faces of the Cayley graph, and these faces correspond precisely to rank 2 parabolic subsystems. The stronger global biclosed hypothesis is used only later in \cref{sec:biclosed} to prove weak separability. 

\begin{remark}
For each set $X\subseteq V$, we have $(\spange X)\cap\AA=\conv(\wh X)$. Therefore, one can reinterpret \cref{def:biclosed,def:parabolic_biclosed} in the language of rescaled roots and convex hulls. For example, a set $R\subseteq\Phi^+$ is closed if and only if for all $\alpha,\beta\in R$, the set $\conv\{\wh\alpha,\wh\beta\}\cap\wh\Phi^+$ is contained in $\wh R$. 
\end{remark} 

A hyperplane $H\subseteq V$ is a \dfn{weakly separating hyperplane}
  for a set $R\subseteq\Phi^+$ if $R\setminus H$ and ${\Phi^+\setminus (R\cup H)}$ are
  contained in opposite sides of $H$. Every weakly separable set has a
  weakly separating hyperplane by the hyperplane separation theorem
  \cite[Theorem 3.3.9]{StoerWitzgall} (see also the discussions in \cite{barkley:2023affine,barkley:2024combinatorial}). Because $V$ is
  3-dimensional, it follows from \cite[Lemma
    5.1]{barkley:2024combinatorial} that a set $R\subseteq \Phi^+$ is
  weakly separable if and only if there exists a weakly separating
  hyperplane $H$ for $R$ such that $H\cap R$ is biclosed in $H\cap
  \Phi^+$. In particular, if $R$ is biclosed, then $R$ is weakly
  separable if and only if there is a weakly separating hyperplane for
  $R$.

\begin{remark}[Dependence on Realization]\label{rem:realizations}
Deformations of the bilinear form $\langle-,-\rangle$ with a compatible deformation of the $U_3$ action on $V$ lead to deformations of the root system $\Phi$. \emph{A priori}, such a deformation may change which subsets of $\Phi^+$ are biclosed or weakly separable. For instance, \cite[Section 6.1]{DyerWang2021} constructs a deformation of the root system of a rank 4 Coxeter group that changes the weakly separable sets. However, the main result of \cite{DyerWang2021} shows that weakly separable sets of rank 3 root systems are invariant under deformations, and hence are an invariant of the Coxeter group in that case. 
For every Coxeter group, Dyer remarked that the property of a set of roots being biclosed is an invariant of the Coxeter group (see \cite[Section 11.2]{dyer:2019weak}). 
Hence the question of comparing biclosed sets with weakly separable sets for the rank 3 root system $\Phi$ is really a Coxeter-theoretic question about $U_3$.
\end{remark}

It is known that the poset of weakly separable sets of positive roots under containment is a lattice \cite[Lemma~2.33, Corollary~2.36]{LabbeThesis}. In general, an analogous statement holds for all root systems of rank $3$, but it can fail for root systems of rank at least $4$.

There are root systems of rank $4$ in which not all biclosed sets
are weakly separable. However, Barkley and Speyer conjectured that biclosed
sets for rank~$3$ Coxeter groups are always weakly separable \cite[Conjecture~1.4]{barkley:2023affine}. (By \Cref{rem:realizations}, this conjecture does not depend on a choice of root system for the group.) This
conjecture implies Dyer's conjecture for rank~$3$ Coxeter groups. It
is this conjecture of Barkley and Speyer that we will resolve for the
universal rank~$3$ Coxeter group.

We end this section with a collection of geometric statements about the root system $\Phi$ that will be used in the proofs of the main theorems.
We begin with a characterization of rank~$2$ parabolic subsystems of $\Phi$.

\begin{proposition}\label{prop:rank2_subsystems}
    Every rank~$2$ linear subsystem of $\Phi$ is infinite. Moreover, a rank~$2$ linear subsystem is parabolic if and only if it is affine.
\end{proposition}

\begin{proof}
    Let $\Phi'$ be a rank~$2$ linear subsystem of $\Phi$.
    Choose $w\in U_3$ such that $w \cdot \Phi'$ contains a simple root, say $\alpha_1$, and let $\beta$ be some other positive root in $w \cdot \Phi'$.
    Since $w$ acts invertibly on $V$, if $w\cdot \Phi'$ is infinite, then so is $\Phi'$.

    Let $H$ be the plane spanned by $\alpha_1$ and $\beta$. 
    The line $\wh{H}$ intersects $\partial\ \conv(\Delta)$ at the vertex $\alpha_1$ and at the opposite side between $\alpha_2$ and $\alpha_3$. 
    Since $\partial \DD$ is the inscribed circle of that triangle, the line $\wh H$ meets $\partial \DD$ at one or two points. 

    The bilinear form associated to the root system $w\cdot \Phi$ is the restriction of $\langle -,-\rangle$ to $H$. 
    Since $H$ contains an isotropic vector, this root system is infinite.

    For the second statement, we observe that the line $\wh{H}$ is tangent to $\partial \DD$ if and only if it includes either $\alpha_2$ or $\alpha_3$.
    In this case, $w\cdot\Phi'$ is the standard parabolic subsystem generated by $\{\alpha_1,\alpha_2\}$ or $\{\alpha_1,\alpha_3\}$.
\end{proof}

\begin{corollary}\label{cor:rank2_subsystem}
    For any two distinct positive roots $\alpha,\beta\in\Phi$, the plane $\Span\{\alpha,\beta\}$ has a nonempty intersection with $\DD$.
\end{corollary} 

\begin{proof}
    The proof of \cref{prop:rank2_subsystems} shows that the given statement is true when $\alpha$ is a simple root. 
    Since the $U_3$ action preserves the isotropic cone, the general result follows.
\end{proof}

In the following lemma, we show that the height statistic on a root is inversely related to the distance of the corresponding rescaled root from the disk.
\cref{lem:rank2geom} then examines the behavior of the height statistic on rank 2 linear subsystems.

\begin{lemma}\label{lem:heightnorm}
    For any positive root $\beta$, we have
    \[ \height{\beta} = \qform(\wh{\beta})^{-\frac{1}{2}}. \] 
\end{lemma} 

\begin{proof}
    Let $\beta\in\Phi^+$.
    Then $\beta= c_1\alpha_1 + c_2\alpha_2 + c_3\alpha_3$ for some $c_1,c_2,c_3\geq 0$, so $\height{\beta} = c_1+c_2+c_3$.
    By definition, $\wh{\beta} = d\cdot \beta$ for some $d>0$ such that $\wh{\beta}\in\AA$.
    This means $d(c_1+c_2+c_3) = 1$.
    Using bilinearity, we find that $\qform(\wh{\beta}) = d^2 \cdot \qform(\beta)$.
    Since $\beta$ is a root, we have $\qform(\beta)=1$.
    Thus,
    \[ \height{\beta} = d^{-1} = \qform(\wh{\beta})^{-\frac{1}{2}}, \]
    as desired.
\end{proof}

\begin{lemma}\label{lem:rank2geom}
    Let $\Phi'=H\cap \Phi$ be a rank~$2$ linear subsystem containing positive roots $\beta,\gamma$. 
    \begin{enumerate}
        \item $\langle\beta,\gamma\rangle>0$ if and only if $\wh{\beta}$ and $\wh{\gamma}$ are on the same side of $\DD\cap\wh{H}$ along the line $\wh{H}$.
        \item If $\wh{\gamma}$ is between $\wh{\beta}$ and $\DD\cap\wh{H}$ along the line $\wh{H}$, then $\height{\beta}<\height{\gamma}$.
        \item If $\beta$ is relatively simple in $\Phi'$ and $\gamma$ is not relatively simple, then $\height{\beta}<\height{\gamma}$. 
    \end{enumerate}
\end{lemma}

\begin{proof}
    (1) Consider the open half-space $\cH=\{\bx : \langle\beta,\bx\rangle>0\}$ containing $\beta$.
    By \cref{cor:rank2_subsystem}, the line $\wh{H}$ has a nonempty intersection with the disk $\DD$. 
    The rescaled root $\wh{\beta}$ is contained in two lines tangent to $\partial \DD$, say at points $p,q$.
    By \cref{prop:affine_orthogonal_subspaces}, the lines $\wh{H}_p,\wh{H}_q$ meet at $\wh{\beta}$, and $\wh{H}_{\beta}$ is the line containing $p$ and $q$.
    Since $\wh{H}$ has nonempty intersection with $\DD$, the line $\wh{H}$ intersects $\wh{H}_{\beta}$ (weakly) between $p$ and $q$ in $\DD$.
    Thus, $\gamma\in\cH$ if and only if $\wh{\gamma}$ is on the same side of $\DD\cap \wh{H}$ as $\wh{\beta}$.

    (2) Suppose $\wh{\gamma}$ is between $\wh{\beta}$ and $\DD\cap\wh{H}$ along the line $\wh{H}$.
    The graph of $\qform:\AA\rightarrow\RR$ is a paraboloid whose level sets are concentric circles centered at $O$.
    Since $\wh{\gamma}$ is closer to the center $O$ than $\wh{\beta}$, we have $\qform(\wh{\beta})<\qform(\wh{\gamma})$.
    Therefore, by \cref{lem:heightnorm}, we have 
    \[ \height{\beta}=\qform(\wh{\beta})^{-\frac{1}{2}}>\qform(\wh{\gamma})^{-\frac{1}{2}} = \height{\gamma}. \qedhere\] 

    (3) We now assume $\beta$ is relatively simple and $\gamma$ is not.
    If $\langle \beta,\gamma\rangle>0$, then ${\height{\beta}<\height{\gamma}}$ by (1) and (2).
 
    Suppose instead $\langle \beta,\gamma\rangle<0$, and let $\alpha\in\Phi'$ be the other relatively simple root of $\Phi'$. 
    The reflection $r_{\alpha}$ permutes the set $\Phi'^+\setminus\{\alpha\}$.
    We have $\langle \alpha,\gamma\rangle>0$, so that $\langle \beta, r_{\alpha}(\gamma)\rangle>0$. 
    Again, from (1) and (2) we have $\height{\beta}\leq \height{r_{\alpha}(\gamma)}$.
    On the other hand, by definition
    \[ r_{\alpha}(\gamma) = \gamma - 2\langle \alpha,\gamma\rangle \alpha, \]
    so $\height{r_{\alpha}(\gamma)}<\height{\gamma}$.
\end{proof}

The remaining statements concern the location of accumulation points of rescaled roots, which are also known as limit roots.

\begin{lemma}[{\cite[Example 2.3]{HohlwegLabbeRipoll}}]\label{lem:limitpointsrank2}
  Let $\Phi'\subseteq \Phi$ be a rank 2 linear subsystem. Let $H$ be
  the hyperplane spanned by $\Phi'$. Each intersection point
  of $\wh H$ with $\partial\DD$ is an accumulation point of ${\widehat{\Phi}'}{}^+$.  
\end{lemma}

\begin{lemma}\label{lem:limitpoints}
  A point in $\AA$ is an accumulation point of $\wh\Phi^+$ if and only if it is in $\partial\DD$.
\end{lemma}

\begin{proof}
  This follows from \cite[Theorem 4.4]{DyerHohlwegRipoll} since $\Phi$ is a hyperbolic root system (in the sense given in that article).
\end{proof}

\section{Snakes and Parabolic Biclosed Sets}\label{sec:parabolic}

Throughout this section, we fix a parabolic biclosed set $R\subseteq\Phi^+$, and we let $B=\Phi^+\setminus R$. Let $\wh R$ and $\wh B$ be the corresponding sets of rescaled roots. Let $E_R=\{e_\beta\mid \beta\in R\}$ and $E_B=\{e_\beta\mid\beta\in B\}$; these sets form a bipartition of the edge set of $\Gamma$. We color the edges in $E_R$ red, and we color the edges in $E_B$ blue. We will
construct two curves called \emph{snakes} on the hyperbolic plane.
The pair of snakes will separate $E_R$ from
$E_B$.  Each snake will begin at the vertex $v_\id$ and then slither along through $\DD$ according to specific local
rules.
When $R$ (respectively, $B$) is finite, the heads of the two snakes meet up with each other so that the snakes enclose the
finitely many edges in $E_R$ (respectively, $E_B$).

It is already known that finite and cofinite biclosed sets are weakly
separable.  See \cref{finite-corollary} for a quick proof in these cases using our upcoming results about snakes.   
When $R$ and $B$ are both infinite, we will show that the two heads of the snakes
converge to limit points on the boundary of $\DD$. 
We consider the hyperplane $H_\snake$ in $V$ defined so that the line $\wh H_\snake=H_\snake\cap\AA$ passes through these two limit points. When $R$ and $B$ are biclosed, we will prove
later that $H_\snake$ is a weakly separating hyperplane for~$R$.

Let $f$ be a face of $\Gamma$. Two edges on the boundary of $f$ form a \dfn{color change} if they share a common endpoint $v$ and they have different colors; in this case, we say the color change \dfn{occurs at} the vertex $v$. 

\begin{lemma}\label{change-at-source}
  Let $f$ be a face of $\Gamma$. There are at most two color changes on the boundary of $f$. Moreover, if there is a color change on the boundary of $f$, then a color change occurs at the source vertex $v_{\yy(f)}$. 
\end{lemma}

\begin{proof}
Let $P$ be the rank 2 parabolic subsystem such that the set of edges on $\partial f$ is $\{ e_{\beta} : \beta \in P \}$. Let $u_0$ be the source vertex of $f$, and number the other vertices of $\partial f$ as in \cref{parabolic-geometry}; number the elements of $P$ as $\{ \beta_j : j \neq 0 \}$ as in \cref{parabolic-geometry}. 

As described in \cref{parabolic-geometry}, the rescaled roots $\wh\beta$ for $\beta \in P \}$ occur on a line segment in the order $\wh\beta_{-1}$, $\wh\beta_{-2}$, $\wh\beta_{-3}$, \dots, $\wh\beta_3$, $\wh\beta_2$, $\wh\beta_1$. Since $B$ is biclosed, $B \cap P$ is either an initial or a final segment of this order; without loss of generality, assume $B \cap P$ is a final segment. (If $B \cap P$ is an initial segment, the roles of red and blue are reversed for the rest of the proof.) We now work through the various possibilities for this final segment. 

If $B \cap P$ is the finite final segment $\wh\beta_j$, \dots, $\wh\beta_2$, $\wh\beta_1$, then the path $u_0 \to u_1 \to \cdots \to u_j$ is blue and the rest of $\partial f$ is red, so there are color changes at $u_0$ and at $u_j$.  

If $B \cap P$ is the cofinite final segment $\wh\beta_{-j-1}$, $\wh\beta_{-j-2}$, \dots, $\wh\beta_3$, $\wh\beta_2$, $\wh\beta_1$, then the path $u_0 \to u_{-1} \to \cdots \to u_{-j}$ is red and the rest of $\partial f$ is blue,  so there are color changes at $u_0$ and at $u_{-j}$. 

If $B \cap P$ is the final segment \dots, $\wh\beta_2$, $\wh\beta_1$, then $u_0 \to u_1 \to u_2 \to \cdots$ is blue and $u_0 \to u_{-1} \to u_{-2} \to \cdots$ is red, so there is one color change at $u_0$. 

Finally, if $B \cap P = P$, then all of $\partial f$ is blue and, if $B \cap P = \varnothing$, then all of $\partial f$ is red. In these cases, there is no color change.
\end{proof}

Suppose $f$ is a face of $\Gamma$ whose boundary has at least one color change. By \cref{change-at-source}, the boundary of $f$ has a color change that occurs at $v_{\yy(f)}$. We will draw a curve $\mathcal S_f$ with no self-intersections that starts at $v_{\yy(f)}$ and is oriented away from $v_{\yy(f)}$. If the boundary of $f$ has exactly one color change (which occurs at $v_{\yy(f)}$), then $\mathcal S_f$ travels through $f$ and approaches the boundary of $\DD$. If the boundary of $f$ has a second color change at some vertex $v_{\zz(f)}$, then $\mathcal S_f$ travels from $v_{\yy(f)}$ to $v_{\zz(f)}$. We call $\mathcal S_f$ a \dfn{snake segment}. (See \cref{fig:snake_segments}.)  

\begin{figure}
    \centering
    \includegraphics[width=0.8\linewidth]{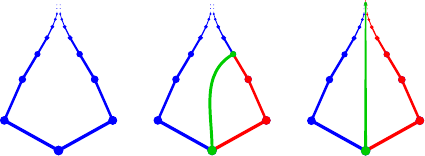}
    \caption{On the left is a face whose boundary has no color change. In the middle is a face whose boundary has two color changes. On the right is a face whose boundary has one color change. Each of the faces in the middle and on the right has a snake segment.} 
    \label{fig:snake_segments} 
\end{figure}

Now suppose $\yy(f)\neq\id$. Then there is another face $f^{(1)}$ that has a color change on its boundary occurring at $v_{\yy(f)}$. Since $\yy(f)\neq\id$, we know that $\yy(f)\neq\yy(f^{(1)})$, so we must have $\yy(f)=\zz(f^{(1)})$. Thus, the curve $\mathcal S_{f^{(1)}}$ travels from $v_{\yy(f^{(1)})}$ to the vertex $v_{\zz(f^{(1)})}=v_{\yy(f)}$, where it touches $\mathcal S_f$. Note that the Coxeter length of $\yy(f^{(1)})$ is strictly less than that of $\yy(f)$. Thus, by repeating this process, we find a sequence of faces $f,f^{(1)},f^{(2)},\ldots,f^{(k)}$ such that $\yy(f^{(k)})=\id$ and such that the union of the snake segments $\cS_{f}\cup\cS_{f^{(1)}}\cup\cdots\cup\cS_{f^{(k)}}$ is a single oriented curve starting at $v_\id$ and ending at $v_{\zz(f)}$. Let $\omega(f)$ be the face $f^{(k)}$ constructed in this procedure. 

\begin{proposition}\label{prop:v1colors}
    If $R$ and $B$ are both nonempty, then the edges incident to $v_\id$ do not all have the same color. 
\end{proposition}

\begin{proof}
    Let $f,f'$ be faces of $\Gamma$ such that $f$ contains an edge in $R$ and $f'$ contains an edge in $B$. 
    By the connectivity of the 1-skeleton (vertices and edges) of the Farey tessellation, there is a sequence of faces and edges $f=f^{(0)},e^{(1)},\ldots,e^{(k)},f^{(k)}=f'$ such that $e^{(i)}$ is incident to $f^{(i-1)}$ and $f^{(i)}$ for all $i$.
    Hence, there is a face $f^{(j)}$ that contains edges of both colors.
    By the construction above, there is a sequence of snake segments going backwards from $f^{(j)}$ to a face $f''$ with source $v_\id$. 
    By definition, the two edges in $f''$ incident to $v_\id$ have different colors.
\end{proof}

\begin{corollary}\label{cor:bicsimple}
    Every nonempty parabolic biclosed subset of $\Phi^+$ contains a simple root.
\end{corollary}

\begin{remark}
    A statement about biclosed subsets of any root system analogous to \cref{cor:bicsimple} appears as \cite[Lemma~1.7(ii)]{dyer:2019weak}.
\end{remark}

For the remainder of this section, we assume that $R$ and $B$ are both nonempty, so that the edges incident to $v_\id$ do not all have the same color by \cref{prop:v1colors}.
Then there are two faces $f_1$ and $f_2$ that have color changes on their boundaries occurring at $v_\id$. 
For each $i\in\{1,2\}$, by construction, the union of all snake segments $\cS_f$ such that $\omega(f)=f_i$ is a single oriented curve $\mathbf{S}_i$ that starts at $v_\id$. 
We call this curve $\mathbf{S}_i$ a \dfn{snake}. 
The union of $\mathbf{S}_1$ and $\mathbf{S}_2$ is a topological $1$-manifold that we call a \dfn{pair of snakes}.

\begin{proposition}\label{prop:manifold}
    A pair of snakes is a connected 1-manifold that separates the disk into two regions.
    Any two edges in the same region have the same color.
\end{proposition}

\begin{proof}
    At each endpoint $v$ of a snake edge, one edge of $\Gamma$ incident to $v$ has the opposite color as the other two edges. 
    Hence, there is exactly one other snake edge that has $v$ as an endpoint.
    Consequently, the union of all snake edges is a 1-manifold.
    By construction, each snake $\mathbf{S}_i$ is path connected and contains a common point $v_\id$, so the pair of snakes is a connected 1-manifold.

    As a connected 1-manifold embedded in the disk, the pair of snakes separates $\DD^{\circ}$ into at most two regions. 
    We claim that any red edge $e_{\alpha}\in E_R$ and blue edge $e_{\beta}\in E_B$ must be in distinct components.
    If not, then there must exist a path $\phi:[0,1]\rightarrow \DD$ from $e_{\alpha}$ to $e_{\beta}$ that avoids the pair of snakes.
    We may require that $\phi$ avoids all vertices of $\Gamma$.
    This path crosses a sequence of edges of $\Gamma$.
    Two consecutive edges $e$, $e'$ along the path must have different colors.
    Thus, $e$ and $e'$ are incident to a common face $f$ with distinct colors, so they are separated within $f$ by a snake edge $\cS_f$.
    Therefore, $\phi$ crosses $\cS_f$, a contradiction.
\end{proof}

\begin{proposition}\label{monotonic-prop}
  Each snake may be chosen within its isotopy class so that for each $\beta\in\Phi^+$, it only passes through $\wh H_{\beta
  }$ 
  at most once. 
\end{proposition}

\begin{proof}
  This follows from the fact 
  that the sequence of vertices of $\Gamma$ through which the snake passes is monotonically increasing in Coxeter length and, more specifically, is a chain in weak order.  So the result follows from the definition of weak order as an order by containment of inversion sets.   
\end{proof}

\begin{lemma}\label{finite-corollary}
Suppose the heads of the snakes meet at a vertex $v^\star$ of $\Gamma$. The edges of $\Gamma$ enclosed by the pair of snakes are exactly the edges on the unique path from $v_\id$ to $v^\star$. 
\end{lemma}

\begin{proof}
Suppose instead that the edges enclosed by the pair of snakes do not form a path. Then there is a vertex $v$ such that all three edges incident to $v$ are enclosed by the pair of snakes. Let $f$ be the face such that $v_{\yy(f)}=v$. Because there is no color change at $v$, neither snake passes through $f$. This implies that $f$ is enclosed by the pair of snakes. This contradicts the assumption that the two heads of the snakes meet at a vertex in $\Gamma$; indeed, the pair of snakes is a Jordan curve, so it does not enclose any points on $\partial\DD$. 
\end{proof}

\begin{proposition}\label{prop:finite_snake}
    Let $R$ be a parabolic biclosed set, and let $B=\Phi^+\setminus R$.
    The following statements are equivalent.
    \begin{enumerate}
        \item Either $R$ or $B$ is finite and nonempty.
        \item There exists $w\in U_3\setminus\{\id\}$ such that either $R=\Inv(w)$ or $B=\Inv(w)$.
        \item There exists $w\in U_3$ such that the heads of the snakes defined by the bipartition $\Phi^+=R\sqcup B$ meet at $v_w$. 
    \end{enumerate}
    Moreover, if these statements are true, the elements $w$ in statements (2) and (3) are the same.
\end{proposition}

\begin{proof}
    Assume (3) holds.
    Then (2) follows from \cref{finite-corollary} and \cref{rem:geodesics}.

    Since inversion sets are finite, (1) follows from (2).

    Assume (1) holds, and suppose $R$ is finite.
    Choose a root $\beta\in R$ of maximum depth.
    Let $v^{\star}$ be the vertex whose lower edge is $e_{\beta}$.
    By assumption, the two upper edges of $v^{\star}$ are in $B$.
    Then $v^{\star}$ is the terminal vertex of two snake edges, so it must be the head of both snakes, and (3) follows.
\end{proof}

\begin{remark}
    The equivalence of statements (1) and (2) in \cref{prop:finite_snake} has an analogue for any Coxeter group by \cite[Lemma~4.1]{dyer:2019weak}.
\end{remark}

\begin{lemma}\label{lem:uniq_limit}
    If $R$ and $B$ are both infinite, then each snake has a unique limit point on $\partial \DD$.
\end{lemma}

\begin{proof}
    First suppose a snake $\mathbf{S}$ is the union of an infinite set of snake segments $\cS_{f^{(1)}}, \cS_{f^{(2)}},\ldots$, ordered so that $\yy(f^{(i+1)}) = \zz(f^{(i)})$ for all $i$.
    By definition of snake segment, the elements $\yy(f^{(1)}),\yy(f^{(2)}),\ldots$ form an infinite chain in weak order.
    As an infinite set of points in the compact set $\DD$, there is an accumulation point for the corresponding set of vertices $v_w$.
    Since the union of the ideal triangles $T_w$ contains the interior of the disk, any accumulation point for a subset of $\{v_w: w\in U_3\}$ must be on $\partial \DD$.
    Since the set $\{\yy(f^{(1)}),\yy(f^{(2)}),\ldots\}$ is a chain in the weak order, this accumulation point is unique, and it must be the limit point of $\mathbf{S}$.

    Now suppose $\mathbf{S}$ is the union of finitely many snake segments $\cS_{f^{(1)}},\ldots,\cS_{f^{(k)}}$, where 
    $\yy(f^{(i+1)})=\zz(f^{(i)})$ for all $i$.
    If the boundary of $f^{(k)}$ has a second color change, then the snake terminates at a vertex $v_{\zz(f^{(k)})}$.
    In this case, the two upper covers of $\zz(f^{(k)})$ have the opposite color as the lower cover of $\zz(f^{(k)})$.
    Hence, there is a second snake terminating at $\zz(f^{(k)})$.
    By \cref{finite-corollary}, we conclude that either $R$ or $B$ is finite, a contradiction.
    Thus, we may assume that $f^{(k)}$ does not have a second color change, so the snake segment $\cS_{f^{(k)}}$ goes to the boundary.        
\end{proof}

\begin{lemma}
  Assume $E_R$ and $E_B$ both contain
  three or more edges of the same depth. Then $R$ and $B$ are both infinite and the snakes $\mathbf{S}_1,\mathbf{S}_2$ limit to distinct points on $\partial\DD$. 
\end{lemma}

\begin{proof}
    By \Cref{finite-corollary} and \Cref{prop:finite_snake}, 
    if $R$ (respectively, $B$) is finite, then $R$ (respectively, $B$) contains at most one edge of each depth. Hence since $E_R$ and $E_B$ each contain two or more edges of the same depth for some choice of depth, we know that $R$ and $B$ both must be infinite.

  Now let $e_1,\dots e_r$ be red edges of $\Gamma$ that all have the
  same depth $d$, with $r\geq 3$. 
    One of these edges, say $e_{\beta}$, has the property that $\wh{H}_{\beta}$ is not crossed by the pair of snakes.
By \cref{prop:manifold}, all tree edges 
on the opposite side of $\wh{H}_{\beta}$ from 
$O$ have the same color as $e_{\beta}$, implying that all points on $\partial \DD $ on the opposite side of $\wh{H}_{\beta }$ from $O$ also have this same color. 
Hence, there is a nontrivial arc of $\partial\DD$ of that color.  We may use this reasoning to deduce that there are nontrivial arcs of $\partial \DD $ of both colors, implying the desired result.
\end{proof}

\begin{remark}\label{same-point-description}
If the pair of snakes both limit to the same point in $\partial \DD $, then one of the following must hold: 
\begin{itemize}
\item The set of edges of $\Gamma$ contains exactly one edge at depth $d$ for each integer $d\geq 0$. 
  \item There exists an integer $j\geq 0$ such that for every integer $d\geq 0$, the number of edges enclosed by the snakes at depth $d$ is $1$ if $d\leq j$ and is $2$ if $d>j$. 
\end{itemize} 
\end{remark}

The preceding remark implies that the biclosed sets having one edge at depth $d$ for all $d$ are in bijection (up to complementation)
  with paths from $v_\id$ to $\partial \DD$. 
  Similarly, the biclosed sets having two edges at distance $d$ from $v_\id$ for all $d>j$ for some $j\ge 0$ are in bijection (up to
  complementation) with the rank 2 parabolic subgroups of $U_3$ with a finite (possibly empty) set of elements deleted, where this finite set corresponds to a path from the source vertex of the parabolic subgroup to an element of the parabolic subgroup. 

\section{Separating Planes and Biclosed Sets}\label{sec:biclosed}

This section is devoted to proving \cref{thm:main}.  To this end, let us now assume that the sets $R$ and $B$ in the bipartition $R\sqcup B=\Phi^+$ are biclosed (not just parabolic biclosed).  We already know that $R$ and $B$ are weakly separable if either one is finite, so we will assume they are both infinite. As above, this bipartition gives rise to two snakes $\mathbf{S}_1$ and $\mathbf{S}_2$. For $i\in\{1,2\}$, let $p_i$ be the limit point of $\mathbf{S}_i$ on $\partial\DD$, and let $\ell_i$ be the line in $\AA$ that is tangent to $\DD$ at $p_i$. 

We continue the convention that edges in $E_R$ are colored red and edges in $E_B$ are colored blue, so that the pair of snakes separates the red edges from the blue edges. Note that $\DD \setminus (\mathbf{S}_1\cup\mathbf{S}_2\cup\{p_1,p_2\})$ has two components.
We say a point in $\partial\DD \setminus \{p_1,p_2\}$ is  \dfn{red} or \dfn{blue} according to the side of the pair of snakes on which it lies. 

\begin{definition}\label{def:Hsnake}
Let $H_\snake$ be the hyperplane in $V$ such that the line $\wh H_\snake=H_\snake\cap\AA$ intersects $\partial\DD$ in the points $p_1$ and $p_2$. If $p_1=p_2$, then $\wh H_\snake$ is tangent to $\partial\DD$.
\end{definition} 

\begin{figure}[]
  \begin{center}
    \includegraphics[height=66.5mm]{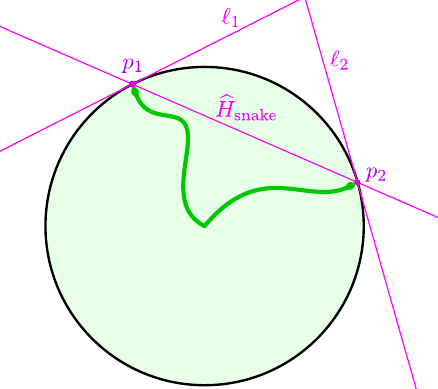}\qquad
    \includegraphics[height=66.5mm]{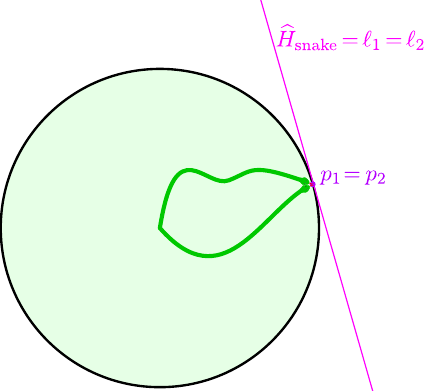}
  \end{center}
\caption{The line $\wh H_\snake$ when $p_1\ne p_2$ (left) and when
  $p_1=p_2$ (right). }
\label{fig-4-1} 
\end{figure}

The \dfn{red side} (respectively, \dfn{blue side}) of $\wh H_\snake$ is the side containing red points (respectively, blue points) of $\partial\DD\setminus\{p_1,p_2\}$. We say a root $\alpha$ is \dfn{red} if $e_\alpha\in E_R$ and \dfn{blue} if $e_\alpha\in E_B$. Our goal is to show that a rescaled root $\wh\alpha \in \wh\Phi^+\setminus \wh H_\snake$ is red if and only if $\wh\alpha$ is on the red side of $\wh H_\snake$.

\begin{lemma}\label{lem:heights}
  For every positive integer $N$, there are finitely many roots $\gamma$ with
  $\mathrm{height}(\gamma)\leq N$. If $\gamma \in \Span_{\geq
    0}\{\alpha,\beta\}$, then $\mathrm{height}(\gamma) >
\min\{  \mathrm{height}(\alpha),\mathrm{height}(\beta)\}$.
\end{lemma}
\begin{proof}
    The first claim follows since there are finitely many tuples $(a_1,a_2,a_3)\in\ZZ^3_{\geq 0}$ such that $a_1+a_2+a_3\leq N$. For the second claim, consider $\conv\{\wh\alpha,\wh\beta\} \setminus \DD$. By \Cref{lem:rank2geom}(ii), if this set is connected, then $\height \alpha < \height\gamma < \height\beta $ or $\height\alpha > \height \gamma > \height\beta$. If the set is disconnected, then $\wh\gamma$ is in the same component as either $\wh\alpha$ or $\wh \beta$. If $\wh\gamma$ is in the same component as $\wh\alpha$, then $\height\alpha<\height\gamma$; if $\wh\gamma$ is in the same component as $\wh\beta$, then $\height\beta<\height\gamma$. 
\end{proof}

\begin{figure}
    \centering
    \includegraphics[height=77.162mm]{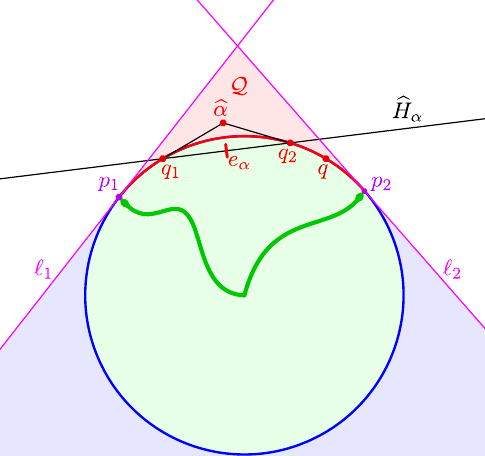}
    \caption{An illustration of the proof of \Cref{lem:snakerootsboundary}.}
    \label{fig:regionoffsnakes}
\end{figure}

\begin{lemma}\label{lem:snakerootsboundary}
  For every point
  $q\in \partial\DD \setminus\{p_1,p_2\}$, there is a neighborhood of
  $q$ in which every rescaled root has the same color as $q$.
\end{lemma}

\begin{proof}
Let $\cQ$ be the connected component of $\AA\setminus (\ell_1\cup \ell_2 \cup \DD^{\circ})$ containing $q$. Every rescaled root $\wh\alpha\in \cQ$ has the property that the hyperplane $\wh H_\alpha$ intersects $\partial\DD$ in two points $q_1,q_2$ such that $\conv(\wh\alpha,q_1,q_2) \cap \partial \DD$ consists of points that are all the same color as $q$ (see \Cref{fig:regionoffsnakes}). By \Cref{monotonic-prop}, if a snake $\mathbf{S}_i$ passed through $\wh H_\alpha$, then its limit point $p_i$ would be in $\conv(\wh\alpha,q_1,q_2)$. This would contradict the monochromaticity of that set, so it cannot happen. Hence, neither snake crosses $\wh H_\alpha$. It follows that $e_\alpha$ (and thus $\alpha$) has the same color as $q_1$ and~$q_2$. 

    We deduce that every $\wh\alpha\in \cQ$ has the same color as $q$. Hence, a neighborhood $\cU$ of $q$ small enough that $\cU\setminus \DD \subseteq \cQ$ will have the desired property.
\end{proof} 

\begin{figure}[]
  \begin{center}
    \includegraphics[height=68mm]{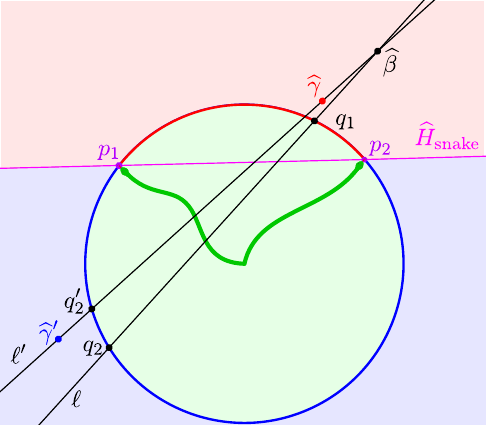}\qquad
    \includegraphics[height=68mm]{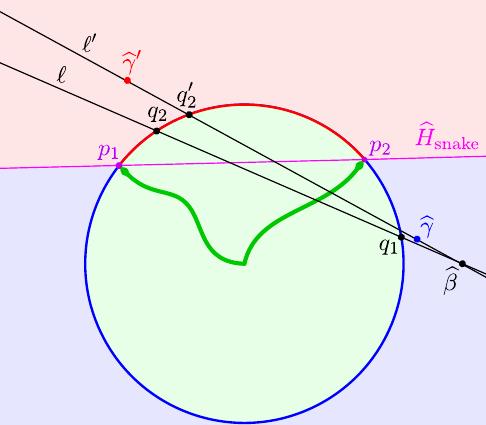}
  \end{center}
\caption{An illustration of the proof of \cref{lem:distinct_separable}.}
\label{fig:distinct}  
\end{figure}

\begin{lemma}\label{lem:distinct_separable}
  If the two snakes $\mathbf{S}_1,\mathbf{S}_2$ limit to
  distinct points $p_1,p_2\in \partial\DD$, then $R$ is weakly
  separated by $H_\snake$. 
\end{lemma}

\begin{proof}
  See \Cref{fig:distinct}. Let $\beta$ be a root not on $H_{\mathrm{snake}}$. Let
  $\overline{p_1p_2}$ be the line segment connecting $p_1, p_2$, and
  let $\ell$ be any line passing through $\beta$ and intersecting
  $\overline{p_1p_2}$ at a single point in its interior. Then $\ell$ intersects
  $\partial\DD$ in two points $q_1,q_2 \not\in \{p_1,p_2\}$ having
  different colors. Let $\wh\gamma$ be a rescaled root close enough to $q_1$ that
  it has the same color as $q_1$ (using
\Cref{lem:limitpoints,lem:snakerootsboundary}) and such that the line
  $\ell'$ passing through $\wh\beta$ and $\wh\gamma$ has a second
  intersection point $q_2'$ with $\partial\DD$ of the same color as $q_2$ (by
  taking $\wh\gamma$ closer to $q_1$, we bring $q_2'$ closer to $q_2$, so
  this is possible). Also take $\wh\gamma$ to be close enough so that
  $\mathrm{height}(\gamma)>\mathrm{height}(\beta)$, using
  \Cref{lem:heights}. Because $\ell'$ contains at least two rescaled roots, there is a rank 2 linear subsystem
  $\Phi'$ of $\Phi$ such that $\wh\Phi'\subseteq\ell'$. Let $\wh\gamma'
  \in \wh\Phi'$ be close enough to $q_2'$ so that (using \Cref{lem:limitpointsrank2,lem:snakerootsboundary}) $\gamma'$ has
  the same color as $q_2'$ and so that $\mathrm{height}(\gamma') >
  \mathrm{height}(\beta)$. By \Cref{lem:heights} and the condition on the heights of
  $\gamma,\gamma'$, we know that $\beta \not\in \mathrm{span}_{\geq
    0}\{\gamma,\gamma'\}$. Hence, we must have either $\gamma \in
  \mathrm{span}_{\geq 0}\{\beta,\gamma'\}$ or $\gamma' \in
  \mathrm{span}_{\geq 0}\{\gamma,\beta\}$. 
  Since $R$ is a biclosed set, in the first case, $\beta$
  is the same color as $\gamma$, and in the second case, $\beta$ is
  the same color as $\gamma'$. In the first case, $\beta$ is on the
  same side of $H_{\mathrm{snake}}$ as $q_1$, and in the second case,
  $\beta$ is on the same side of $H_{\mathrm{snake}}$ as
  $q_2'$. Hence, in both cases, $\beta$ has the color prescribed by
  $H_{\mathrm{snake}}$.  
\end{proof}

\begin{figure}
    \centering
    \includegraphics[width=0.45\linewidth]{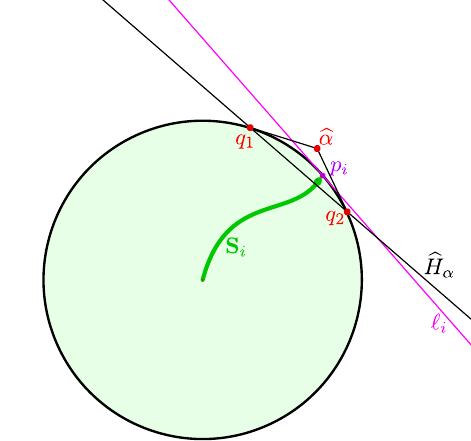}\qquad\includegraphics[width=0.45\linewidth]{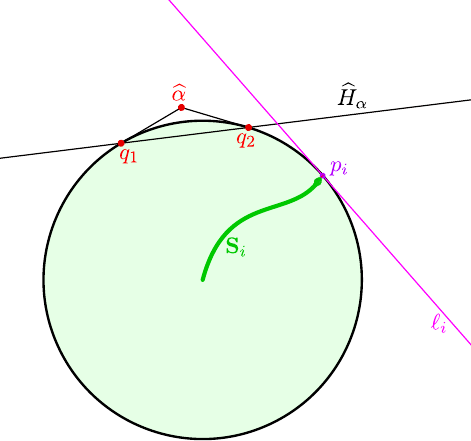}
    \caption{The two cases in the proof of \Cref{lem:rootcrossingsnake}.}
    \label{fig:rootcrossingsnake}
\end{figure}

\begin{lemma}\label{lem:rootcrossingsnake}
    Let $i\in\{1,2\}$. For a rescaled root $\wh\alpha \in \wh\Phi\setminus \ell_i$, the dual hyperplane $\wh H_\alpha$ is crossed by $\mathbf{S}_i$ if and only if $\wh\alpha$ is separated from $\DD$ by $\ell_i$. 
\end{lemma}
\begin{proof}
Let $\wh{\alpha}$ be a rescaled root. Consider the lines containing $\wh{\alpha}$ that are tangent to $\partial\DD$; let $q_1,q_2$ be the points at which these lines are tangent to $\DD$.
    
    If $\wh\alpha$ is separated from $\DD$ by $\ell_i$, then the triangle $\conv(\wh \alpha, q_1,q_2)$ contains $p_i$ 
    and does not contain the center of $\DD$, so the snake $\mathbf{S}_i$ must cross $\wh H_\alpha$ (the third side of the triangle). (See the left side of \cref{fig:rootcrossingsnake}.)

    If instead $\wh\alpha$ is on the same side of $\ell_i$ as $\DD$, then the triangle $\conv(\wh\alpha,q_1,q_2)$ does not contain $p_i$ or the center of $\DD$. Hence, $\mathbf{S}_i$ must not pass through $\wh H_\alpha$. (See the right side of \cref{fig:rootcrossingsnake}.)
\end{proof}

\begin{lemma}\label{same-endpoint}
    If the two snakes $\mathbf{S}_1,\mathbf{S}_2$ limit to the same point $p\in \partial \DD$, then $R$ is weakly separated by $H_\snake$. 
\end{lemma}

\begin{proof}
    Without loss of generality, we may assume $\DD$ is contained in the blue side of $\wh H_\snake$. Equivalently, the red edges are those encircled by $\mathbf{S}_1\cup \mathbf{S}_2$. First observe that, by \Cref{lem:rootcrossingsnake}, a rescaled root $\wh \alpha \in \wh\Phi\setminus \wh H_\snake$ is on the red side of $\wh H_\snake$ if and only if both snakes $\mathbf{S}_1,\mathbf{S}_2$ cross $\wh H_\alpha$. By \Cref{same-point-description}, there are two cases. 
    
    In the first case, there is exactly one red edge of each depth. In this case, a hyperplane $H_\beta$ is crossed by both snakes if and only if the corresponding root $\beta$ is red (see \Cref{fig:5.6}). Hence, every rescaled root on the red side of $\wh H_\snake$ is red, and every rescaled root on the blue side of $\wh H_\snake$ is blue.
    
    In the second case, there is some $w\in U_3$ such that $H_\snake$ is spanned by the roots $\alpha,\beta$, where $e_\alpha,e_\beta$ are the outgoing edges from $v_w$. Let $f$ be the corresponding face. Then the red edges are those on the path from $v_\id$ to $v_w$, together with a subset of the edges on the boundary of $f$. Hence, the edges whose corresponding hyperplanes are crossed by a snake (and which are not edges of $f$) are exactly those on the path from $v_\id$ to $v_w$. A rescaled root $\wh \gamma$ is in $\wh H_\snake$ if and only if $e_\gamma$ is an edge on the boundary of $f$ by \cref{new-parabolic-prop}. The hyperplanes corresponding to the rescaled roots on the red side of $\wh H_\snake$ are crossed by the two snakes and are therefore on the path from $v_\id$ to $v_w$; in particular, they are red. Hyperplanes corresponding to rescaled roots on the blue side of $\wh H_\snake$ are crossed by neither snake, so they are blue. 
\end{proof}

\begin{figure}
    \centering
    \includegraphics[width=0.7\linewidth]{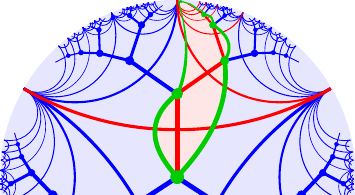}
    \caption{An illustration of the proof of \Cref{same-endpoint}.}
    \label{fig:5.6}
\end{figure}

\begin{proof}[Proof of \cref{thm:main}]
    Let $R$ be a biclosed set of positive roots, and set $B=\Phi^+\setminus R$.
    If $R$ or $B$ is empty, then $\AA_0$ is a separating hyperplane.
    Otherwise, if $R$ or $B$ is finite, then $R$ or $B$ is the inversion set of some $w\in U_3$ by \cref{prop:finite_snake}, which implies $R$ and $B$ are separated by $w\cdot \AA_0$.
    Finally, if $R$ and $B$ are both infinite, then the hyperplane $H_{\snake}$ from \cref{def:Hsnake} weakly separates $R$ and $B$ by \cref{lem:distinct_separable} and \cref{same-endpoint}.
\end{proof}

\section{When is a parabolic biclosed set also biclosed?}\label{sec:When?}
\begin{figure}
    \centering
    \includegraphics[width=0.45\linewidth]{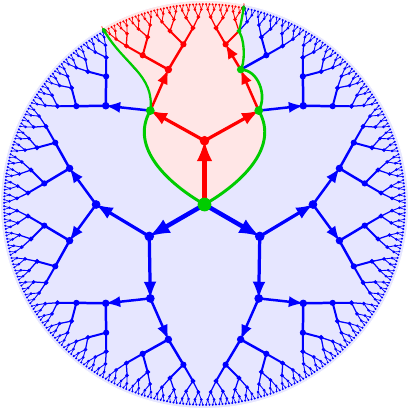}\qquad\includegraphics[width=0.45\linewidth]{SnakesPIC23}
    \caption{Two parabolic biclosed sets whose snakes have the same limit points. Only the set on the right is biclosed.}
    \label{fig:parabolicnotbiclosed}
\end{figure}
In \Cref{sec:parabolic}, we showed that any partition $\Phi^+ = R\sqcup B$ into two parabolic biclosed sets has an associated pair of snakes $\mathbf{S}_1,\mathbf{S}_2$. If $R$ and $B$ are both infinite, then the snakes have limit points $p_1,p_2\in \partial\DD$. In \Cref{sec:biclosed}, we showed that if $R$ and $B$ are biclosed, which is stronger than being just parabolic biclosed, then $R$ and $B$ are weakly separated by a hyperplane through $p_1$ and $p_2$. It is possible for a parabolic biclosed set not to be biclosed at all, much less weakly separable; \Cref{fig:parabolicnotbiclosed} gives such an example. In this section, we will show how to determine, in terms of snakes, whether a parabolic biclosed set is biclosed. 

Let $\Phi^+ = R \sqcup B$ be a partition of the positive roots into parabolic biclosed sets. As usual, we call elements of $R$ red and elements of $B$ blue. If $R$ or $B$ is finite, then both $R$ and $B$ are biclosed by \Cref{prop:finite_snake}. If $R$ and $B$ are both infinite, then the snakes $\mathbf{S}_1,\mathbf{S}_2$ limit to points $p_1,p_2\in \partial\DD$. It turns out the only roots that can fail to be weakly separated by the line through $p_1$ and $p_2$ are those whose reflecting hyperplanes are crossed by a snake. To make this precise, we introduce the following notion. Let $E$ be the set of undirected edges of the Cayley graph $\Gamma$, and let $F$ be the set of faces of $\Gamma$.

\begin{figure}
    \centering
    \includegraphics[width=0.51\linewidth]{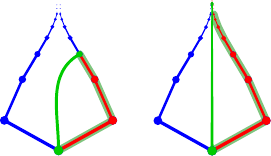}
    \caption{On the left is a face whose boundary has two color changes, with its associated snake segment. The edges between the two snake segment endpoints are in the jagged snake, as indicated by green shading. On the right is a face $f$ whose boundary has one color change, with its associated snake segment; 
    we may arbitrarily pick one of the two halves to be in the jagged snake.} 
    \label{fig:jagged_segments} 
\end{figure}

\begin{definition}
    Let $\mathbf{S}_i$ be one of the two snakes for $R\sqcup B$. The \dfn{jagged snake} associated to $\mathbf{S}_i$ is the set $\bm{\Sigma}_i$ of edges of the Cayley graph defined by the rules in \Cref{fig:jagged_segments}. 
\end{definition}

\begin{figure}
    \centering
    \includegraphics[width=0.45\linewidth]{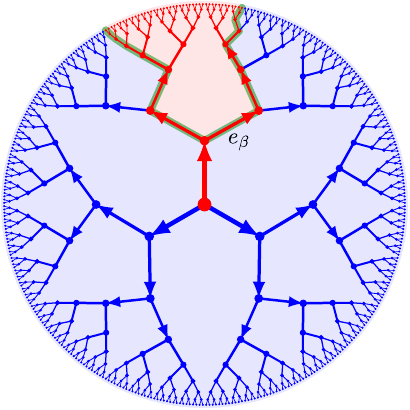}\qquad\includegraphics[width=0.45\linewidth]{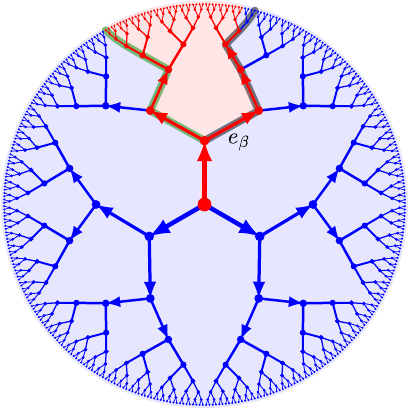}
    \caption{On the left, the jagged pair $\bm\Sigma$ for the parabolic biclosed sets in \Cref{fig:parabolicnotbiclosed}, with an edge $e_\beta$ indicated. On the right, we show $\bm\Sigma_{e_\beta}$ in green and $t_\beta\cdot \bm\Sigma_{e_\beta}$ in black. Since $t_\beta\cdot \bm\Sigma_{e_\beta}$ diverges from $\bm\Sigma$ into the area shaded blue, \Cref{thm:jaggedtest} says that $e_\beta$ must be blue in any biclosed set having $\bm\Sigma$ as a jagged pair. In particular, the set depicted on the left of \Cref{fig:parabolicnotbiclosed} is not biclosed.}
    \label{fig:jaggedexample}
\end{figure}

The jagged pairs for the two parabolic biclosed sets in \Cref{fig:parabolicnotbiclosed} are the same and are shown in \Cref{fig:jaggedexample}. 
Let $\bm\Sigma$ be a jagged pair for $R\sqcup B$. 

\begin{lemma}\label{lem:jaggedempty}
    If $\bm\Sigma$ is empty, then $R$ and $B$ are biclosed.
\end{lemma}
\begin{proof}
    The following are the possible situations when $\bm\Sigma=\varnothing$:
    \begin{itemize}
        \item $R$ or $B$ is finite;
        \item $R$ and $B$ are both infinite, and there is one edge of each depth between $\mathbf{S}_1$ and $\mathbf{S}_2$.
    \end{itemize}
    In the first case, \Cref{prop:finite_snake} implies that $R$ or $B$ is an inversion set, in which case it is biclosed. In the second case, without loss of generality, $R$ contains one root of each depth. Then the subset of $R$ consisting of roots with depth at most $d$ is an inversion set, so $R$ is the union of an increasing family of biclosed sets. Such unions are biclosed. 
\end{proof}

For an element $u \in U_3$ and an edge $e\in E$ with endpoints $w,ws_i$, we define $u\cdot e$ to be the edge with endpoints $uw,uws_i$. In other words, if $\beta$ and $\pm u\beta$ are positive roots, then we set $u\cdot e_{\beta} = e_{\pm u\beta}$. 
This action extends to an action on subsets of $E$, so in particular $u\cdot \bm\Sigma$ is well-defined.

\begin{lemma}\label{lem:jaggedonline}
    Assume $R$ and $B$ are both infinite. Let $\wh H_{\snake}$ be the line through $p_1$ and $p_2$ (or the tangent line at $p_1$ if $p_1=p_2$). If $\bm\Sigma\neq \varnothing$, then a rescaled root $\wh \alpha$ is in $\wh H_{\snake}$ whenever $t_\alpha\cdot \bm\Sigma = \bm\Sigma$. If $\bm\Sigma\neq\varnothing$ and $\wh H_\snake \cap \wh\Phi$ contains at least two rescaled roots, then $\wh \alpha$ is in $\wh H_{\snake}$ if and only if $t_\alpha\cdot \bm\Sigma = \bm\Sigma$.
\end{lemma}
\begin{proof}
    Let $\alpha$ be a positive root. Observe that $\{p_1,p_2\} = \{t_\alpha p_1,t_\alpha p_2\}$ if and only if $t_\alpha H_\snake = H_\snake$ if and only if either $\alpha\in H_\snake$ or $H_\snake = H_\alpha$. (The last equivalence is because subspaces preserved by $t_\alpha$ are sums of eigenspaces for $t_\alpha$.) If $H_\snake = H_\alpha$ and $\bm\Sigma \neq \varnothing$, then $p_1$ is the limit point of a face $f$ of the Cayley graph containing $e_\alpha$, and $\bm\Sigma$ contains some but not all edges of $f$. Hence in this case, $t_\alpha \bm\Sigma\neq \bm\Sigma$. So if $t_\alpha\bm\Sigma = \bm\Sigma$, then $\alpha\in H_\snake$. Conversely, assume that $\bm\Sigma\neq \varnothing$ and that $\alpha$ and at least one other root are in $H_\snake$. Then $t_\alpha$ preserves $\{p_1,p_2\}$ and $e_\alpha \in \bm\Sigma$. If $p_1=p_2$, then $\bm\Sigma$ must be the set of edges of the face of the Cayley graph with limit point $p_1$. This is the set of edges dual to the roots in $H_\snake$, so $t_\alpha \bm\Sigma = \bm\Sigma$. If instead $p_1\neq p_2$, then the fact that there are two roots in $H_\snake$ implies that neither $p_1$ nor $p_2$ is the limit point of a face of the Cayley graph. Then for each vertex $w\in \Gamma$, there is a unique path in the Cayley graph with endpoint $w$ and limit point $p_1$ (respectively, $p_2$). If $e_\alpha$ has endpoints $w,ws_i$, then $t_\alpha$ exchanges the path from $ws_i$ to $p_1$ with the path from $w$ to $p_2$. Since $\bm\Sigma$ is the union of those two paths and $\{e_\alpha\}$, it follows that $t_\alpha\bm\Sigma =\bm\Sigma$, as desired.
\end{proof}

We are now ready to characterize the parabolic biclosed sets that are biclosed. If $e\in \bm\Sigma$, then $e$ is in exactly one of the two jagged snakes, say $e\in \bm\Sigma_1$. In this case, we write $\bm\Sigma_e$ to mean the symmetric difference of $\bm\Sigma_2$ with the subset of $\bm\Sigma_1$ consisting of edges with depth at most the depth of $e$; this is the portion of $\bm\Sigma$ reachable from $e$ by going down (and then later going back up). 

\begin{proposition}\label{thm:jaggedtest}
    Let $\Phi^+ = R\sqcup B$ be a partition of the positive roots into two parabolic biclosed sets, with jagged pair $\bm\Sigma$. Then $R$ and $B$ are biclosed if and only if the following both hold:
    \begin{itemize}
        \item[(a)] For each edge $e_\beta\in \bm\Sigma$ such that $t_\beta \cdot \bm\Sigma\neq \bm\Sigma$, let $e'$ be any edge in $(t_\beta \cdot \bm\Sigma_e) \setminus \bm\Sigma$. Then $e_\beta$ must have the same color as $e'$.
        
        \item[(b)] If there are distinct edges $e_\alpha,e_\beta\in\bm\Sigma$ such that $t_\alpha\cdot \bm\Sigma = t_\beta\cdot \bm\Sigma = \bm\Sigma$, then the restriction of $R$ and $B$ to $\Phi^+\cap \Span\{\alpha,\beta\}$ must be biclosed in $\Span\{\alpha,\beta\}$.
    \end{itemize}
\end{proposition}
\begin{proof}
    By \Cref{lem:jaggedempty}, we may assume $R$ and $B$ are both infinite and that $\bm\Sigma \neq \varnothing$. Let the limit points of $\mathbf{S}_1,\mathbf{S}_2$ be $p_1$ and $p_2$, and $\wh H_\snake$ be the line passing through $p_1$ and $p_2$. By \Cref{lem:jaggedonline}, condition (b) is equivalent to $R\cap H_\snake$ being biclosed in $\Phi^+\cap H_\snake$. Hence it is enough to show that condition (a) is equivalent to $H_\snake$ being a weakly separating hyperplane for $R$. 

    We first treat the rare situation where both $p_1$ and $p_2$ are limit points of faces $f_1,f_2$ of the Cayley graph such that $p_1\neq p_2$ and there exists a (necessarily unique) root $\alpha$ such that $t_\alpha p_1 = p_2$. In this case, we might have a jagged pair $\bm\Sigma$ such that $e_\alpha\in \bm\Sigma$ and $t_\alpha\cdot \bm\Sigma \neq \bm\Sigma$. Assume this holds. Note that condition (a) then imposes a restriction on $\alpha$, despite the fact that $\alpha\in H_\snake$. However, it is easy to verify that $H_\snake$ does not weakly separate $R$ in this situation, since $t_\alpha$ preserves the two halves of $V\setminus H_\snake$. As a result, condition (a) must fail for some edge $e_\beta$ with $\beta\neq \alpha$, as we shall see below. Since this is the only situation where condition (a) might be imposed on a root in $H_\snake$, it will be enough to consider only roots not in $H_\snake$ when showing that biclosed sets satisfy (a).
    
    If this rare situation does not hold, or if it does hold and $t_\alpha\cdot\bm\Sigma=\bm\Sigma$, then \Cref{lem:jaggedonline} implies that $\beta\in H_\snake$ if and only if $e_\beta\in \bm\Sigma$ and $t_\beta\bm\Sigma\neq \bm\Sigma$. So let $\wh \beta$ be a rescaled root that is not on $\wh H_\snake$. Using the language from \Cref{sec:biclosed}, without loss of generality, $\wh \beta$ is on the red side of $\wh H_\snake$. We will show that condition (a) holds for $\beta$ if and only if $\beta$ is red.  Let $\ell_1$ and $\ell_2$ be the tangent lines to $\partial\DD$ at $p_1$ and $p_2$, respectively. 
    
    If $\wh H_\beta$ is crossed by both $\mathbf{S}_1$ and $\mathbf{S}_2$, then $e_\beta$ is in both $\bm\Sigma_1$ and $\bm\Sigma_2$, so $e_\beta\not\in \bm\Sigma$ and (a) holds vacuously. Furthermore, $\wh\beta$ is separated from $\DD$ by both $\ell_1$ and $\ell_2$, by \Cref{lem:rootcrossingsnake}. Let $q_1$ and $q_2$ be the points of tangency for the two tangent lines to $\partial\DD$ that pass through $\wh \beta$. Then both snakes pass through the line segment $\overline{q_1q_2}\subseteq \wh H_\beta$, so $p_1$ and $p_2$ are both in the triangle $\conv\{\wh \beta,q_1,q_2\}$. Hence the red points of $\partial\DD$ are contained in $\conv\{\wh\beta,q_1,q_2\}$. Since $e_\beta$ is in the same component of $\conv\{\wh\beta,q_1,q_2\}\setminus (\mathbf{S}_1\cup\mathbf{S}_2)$ as the red points of $\partial\DD$, $\beta$ is red.

    If $\wh H_\beta$ is crossed by neither snake, then $e_\beta$ is not in $\bm\Sigma$ and (a) holds vacuously. Then, similar to the previous case, we construct the triangle $\conv\{\wh\beta,q_1,q_2\}$ and deduce that all points of $\partial\DD$ in the triangle are red. (This case is illustrated in \Cref{fig:regionoffsnakes}.) Since $e_\beta$ itself has an endpoint in $\conv\{\wh\beta,q_1,q_2\}$ and the snakes do not enter the triangle, $\beta$ is also red.

    The remaining case is when $\wh H_\beta$ is crossed by a single snake, say $\mathbf{S}_1$, so that $e_\beta \in \bm\Sigma_1\setminus\bm\Sigma_2$. Then $e_\beta\in\bm\Sigma$. Since $\wh \beta$ is not in $\wh H_\snake$, we must have $t_\beta\cdot \bm\Sigma \neq \bm\Sigma$ by \Cref{lem:jaggedonline}. 
    We will now show that (a) holds for $\beta$ if and only if $\beta$ is red. The key observation is that every edge in the path $t_\beta\cdot \bm\Sigma_{e_\beta}\setminus \bm\Sigma$ has the same color as the limit point of that path, which is $t_\beta p_2$, and that furthermore $t_\beta p_2$ is on the red side of $\wh H_\snake$ (see \Cref{fig:tp2}). It follows that $\beta$ is red if and only if it is the same color as the edges of $t_\beta\cdot\bm\Sigma_{e_\beta}\setminus \bm\Sigma$, which is condition (a).
\end{proof}

\begin{figure}
\begin{center}
    \includegraphics[height=70mm]{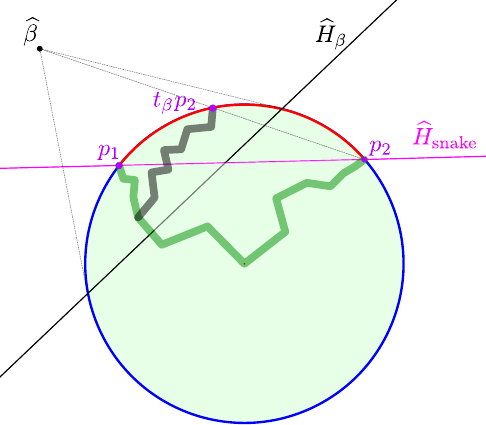}
  \end{center}
    \caption{Illustration of the proof of \Cref{thm:jaggedtest}. }
    \label{fig:tp2}
\end{figure}

\section{Parametrization of Extended Weak Order}\label{sec:structure}

\subsection{Parameterizing Biclosed Sets} 

Let $R$ be a biclosed set. 
By \Cref{thm:main}, $R$ is weakly separable. 
If $R$ is finite (respectively, cofinite), then there is a unique $w\in U_3$ such that $R=\Inv(w)$ (respectively, $R=\Phi^+\setminus\Inv(w)$). 
In either case, there are many weakly separating hyperplanes for $R$.
Among them there are distinguished hyperplanes specified by the following lemma.

\begin{lemma}\label{lem:finite_weak_hyp}
    Suppose $R=\Inv(w)$ for some $w\in U_3\setminus\{\id\}$. 
    Let $\beta_1,\beta_2$ be the roots labeling edges incident to $v_w$ not enclosed by the pair of snakes, and let $\alpha$ be the root labeling the remaining edge incident to $v_w$. Then $H=\Span\{\beta_1,\beta_2\}$, $H_1=\Span\{\beta_1,\alpha\}$, and $H_2=\Span\{\beta_2,\alpha\}$ are weakly separating hyperplanes for $R$.  
    Moreover, $H,H_1,H_2$ are exactly the weakly separating hyperplanes for $R$ that are tangent to $\partial\DD$. 
\end{lemma}

\begin{proof}
    Consider the set $R'=R\cup\spange{\{\beta_1,\beta_2\}}$.
    The roots in $R$ are the labels on edges from $v_\id$ to $v_w$, and the remaining roots of $R'$ are the labels on edges incident to the region $f$ where $\yy(f)=w$.
    From the proof of \cref{same-endpoint}, the hyperplane $H$ spanned by $\{\beta_1,\beta_2\}$ is a weakly separating hyperplane for $R'$.
    Since $R'\cap H=\overline{\{\beta_1,\beta_2\}}$, we deduce that $H$ is a weakly separating hyperplane for $R$ and $R\cap H=\varnothing$. The argument for $H_1$ and $H_2$ is similar, using $R'=R\cup\spange\{\beta_1,\alpha\}$ and $R'=R\cup\spange\{\beta_2,\alpha\}$, respectively.

    The lines $\wh H_1, \wh H_2$ are tangent to $\partial\DD$; call the intersection points $p_1$ and $p_2$, respectively.
    We observe that $\wh{\alpha}$ is separated from $\DD$ by a tangent line $\ell_p,\ p\in\partial \DD$ if and only if $p$ is in the triangle $\conv(\wh\alpha, p_1,p_2)$.
    Hence, any weakly separating line for $\wh{R}$ that is tangent to $\partial \DD$ must meet $\partial \DD$ inside this triangle.
    To ensure $\wh{\alpha}$ is not on $\ell_p$, the point $p$ must be strictly between $p_1,p_2$.
    Let $q\in\partial\DD$ be the point where $\wh{H}$ meets $\partial \DD$.
    Then for each $i\in\{1,2\}$, for any $p$ along the arc strictly between $p_i,q$, the line $\ell_p$ separates $\wh{\beta}_i$ and $\DD$.
    Hence, none of these lines are weakly separating lines for $\wh{R}$. The remaining tangent lines intersecting $\partial \DD$ in the triangle are $\wh H, \ell_{p_1}=\wh H_1,$ and $\ell_{p_2}=\wh H_2$.
\end{proof}

\begin{lemma}\label{lem:inf_weak_hyp}
    If $R$ is neither finite nor cofinite, then the weakly separating hyperplane for $R$ is unique.
\end{lemma}

\begin{proof}
    Suppose to the contrary that there are distinct weakly separating hyperplanes $H,H'$ for $R$.
    Using \cref{lem:limitpoints}, for any $\epsilon>0$ there are a finite number of roots $\alpha$ such that the distance between $\wh{\alpha}$ and $\DD$ is at least $\epsilon$.
    Hence, $\wh{H}$ and $\wh{H}'$ both intersect the disk $\DD$.
    Since the lines are distinct, they do not have the same set of intersection points with $\partial \DD$.
    Suppose $q\in \wh{H}\cap\partial \DD \setminus \wh{H}'$.
    By applying \cref{lem:snakerootsboundary} to $H'$, there is a neighborhood of $q$ in which either every rescaled root is in $\wh{R}$ or every rescaled root is not in $\wh{R}$.
    Without loss of generality, we assume the latter case.
    Since every rescaled root in some neighborhood of $q$ is weakly on the same side of $\wh{H}$ and $q\in \wh{H}$, the density of rescaled roots on $\partial \DD$ (\cref{lem:limitpoints}) implies that $\wh{H}$ is tangent to $\partial \DD$ at $q$. Moreover, $R\setminus H$ is in the half-space $\cH$ defined by $H$ not containing $\DD$.
    But then $R\cap H$ and $R\cap \cH$ are finite, which implies $R$ is finite, a contradiction.
\end{proof}

For a biclosed set $R$ which is neither finite nor cofinite, we label the unique weakly separating hyperplane furnished by \Cref{lem:inf_weak_hyp} as $H_R$. 
Let $\cH_R$ denote the component of
$V\setminus H_R$ that contains elements of $R$. 
There may or may not be roots contained in $H_R$; let $\Phi_R$ denote the (possibly empty) root subsystem $H_R\cap \Phi$. 
Then $R$ is determined by the half-space $\cH_R$ along with the set $\partial R\coloneqq R\cap\Phi_R$. 
Conversely, each hyperplane $H\subset V$ such that $\wh H\cap
\partial\DD\neq\varnothing$ is the weakly separating hyperplane for
some biclosed set. Choosing a biclosed set with $H$ as a weakly separating
hyperplane is equivalent to choosing a half-space in $V\setminus H$
and a biclosed subset of $H\cap \Phi^+$. For a finite or cofinite biclosed set there are three possible such representations by \Cref{lem:finite_weak_hyp}; otherwise, there is a unique such representation. 

To summarize, the following is a parametrization of the biclosed sets in $\Phi^+$.
\begin{itemize}
\item For each $w\in U_3$, there is a finite biclosed set $R=\Inv(w)$. 
\item For each point $p\in \partial\DD$, let $\wh H$ be the tangent
  line through $p$ in $\AA$, and let $H$ be the linear hyperplane in $V$
  containing $\wh H$. For each of the two half-spaces $\cH$ in
  $V\setminus H$, and for each biclosed subset $\partial R$ of $H\cap \Phi^+$, there is a biclosed set $R=(\cH\cap \Phi^+) \sqcup
  \partial R$. If $H\cap \Phi^+$ has rank 2 and $\cH$ is the half-space so that $\cH\cap \DD=\varnothing$ (equivalently, $\cH\cap \Phi^+$ is finite), then we require that $\partial R $ is infinite. If $H\cap \Phi^+$ has rank 2 and $\cH$ is the half-space so that $\cH\cap \DD\neq \varnothing$ (equivalently, $\Phi^+\setminus (H\cup\cH)$ is finite), then we require that $\partial R $ is not cofinite.   
\item For each pair of distinct points $p,q\in \partial\DD$, let $\wh
  H$ be the line through $p$ and $q$ in $\AA$, and let $H$ be the linear
  hyperplane in $V$ containing $\wh H$. For each of the two
  half-spaces $\cH$ in $V\setminus H$, and for each biclosed subset
  $\partial R$ of $H\cap \Phi^+$, there is a biclosed set
  $R=(\cH\cap \Phi^+) \sqcup \partial R$.
\item For each $w\in U_3$, there is a cofinite biclosed set
 $R=\Phi^+\setminus\Inv(w)$.
\end{itemize}

\subsection{Cover Relations} 
\label{coverrelations}

We note the following description of cover relations in the extended
weak order of $U_3$. 

\begin{proposition}
  A containment $R_1\subset R_2$ of biclosed sets is a cover relation
  in the extended weak order of $U_3$ if and only if $|R_2\setminus R_1|=1$.
\end{proposition}

\begin{proof}
  Let $R_1\subset R_2$. Evidently, if $|R_2\setminus R_1|=1$, then
  $R_1\subset R_2$ is a cover relation. To see the converse, assume
  that $|R_2\setminus R_1|\geq 2$. Let $\alpha,\beta \in R_2\setminus
  R_1$ be distinct roots.  Let $K$ be a non-archimedean ordered field
  containing $\RR$ (meaning that there exists an element of $K$
  greater than every natural number). By \cite[Lemma
    2.1.7]{BarkleyThesis}, for any weakly separable set $R\subseteq
  \Phi^+$, there is an $\RR$-linear function $\chi\in
  \mathrm{Hom}_{\RR}(V,K)$ such that $\chi(R) < 0$ and $\chi(\Phi^+\setminus
  R) > 0$. Let $\chi_1,\chi_2$ be such separating functions for $R_1$ and
  $R_2$, respectively. There exists a unique $\lambda_\alpha\in K$
  such that $\lambda_\alpha>0$ and $\lambda_\alpha \chi_1(\alpha) +
  (1-\lambda_\alpha)\chi_2(\alpha)=0$. Set $\chi_\alpha \coloneqq
  \lambda_\alpha \chi_1 + (1-\lambda_\alpha) \chi_2$. Similarly define
  $\lambda_\beta$ and $\chi_\beta$. If $\lambda_\alpha < \lambda_\beta$,
  then $\{\gamma\in \Phi^+ \mid \chi_\alpha(\gamma)<0\}$ is a biclosed
  set between $R_1$ and $R_2$ that contains $\alpha$ and not
  $\beta$. If instead $\lambda_\alpha > \lambda_\beta$, then
  $\{\gamma\in \Phi^+ \mid \chi_\beta(\gamma)<0\}$ is a biclosed set
  between $R_1$ and $R_2$ that contains $\beta$ and not $\alpha$. The
  remaining case is when $\lambda_\alpha=\lambda_\beta$. If one
  chooses $\chi_1$ and $\chi_2$ generically, then this case does not
  happen. Alternatively, one can allow this case to happen. Then
  consider the root subsystem $\Phi' = \{\gamma\in \Phi \mid
  \chi_\alpha(\gamma)=0\}$. If $\Phi'$ has rank 3, then $\chi_1$ and $\chi_2$
  are $K$-multiples of one another; this implies that
  $R_1=\varnothing$ and $R_2=\Phi^+$, so any finite biclosed set
  containing exactly one of $\{\alpha,\beta\}$ will be between $R_1$
  and $R_2$. Otherwise, $\Phi'$ is the rank 2 subsystem spanned by
  $\alpha,\beta$. Then there is a biclosed subset $R'$ of $(\Phi')^+$
  such that $\alpha\in R'$ and $\beta\not\in R'$. We find that $\{\gamma
  \in \Phi^+ \mid \chi_\alpha(\gamma) < 0\} \cup R'$ is a biclosed set
  between $R_1$ and $R_2$ that contains $\alpha$ and not $\beta$.
\end{proof}

Note that adding a single root preserves the property of being finite
or of being cofinite. Furthermore, adding a single root cannot change
$H_R$ for a biclosed set $R$ that is neither finite nor cofinite. As
a result, the following is a complete list of cover relations in
extended weak order:
\begin{itemize}
\item For each $w\in U_3$ and each $s\in S$ such that $\ell(ws) =
  \ell(w)+1$, there is a cover relation $\Inv(w) \subset \Inv(ws)$.
\item For each biclosed set $R$ that is neither finite nor cofinite
  such that $\Phi_R\neq \varnothing$, and for each $\alpha \in
  \Phi_R\setminus R$ such that $\partial R\cup\{\alpha\}$ is biclosed
  in $\Phi_R^+$, there is a cover relation $R\subset R\cup\{\alpha\}$.
\item For each $w\in U_3$ and each $s\in S$ such that $\ell(ws) =
  \ell(w)-1$, there is a cover relation $(\Phi^+\setminus\Inv(w)) \subset (\Phi^+\setminus\Inv(ws))$. 
\end{itemize}

\subsection{Finitely Generated Biclosed Sets and Complete Join-Irreducibles} 

\begin{figure}[]
  \begin{center}
    \includegraphics[height=70mm]{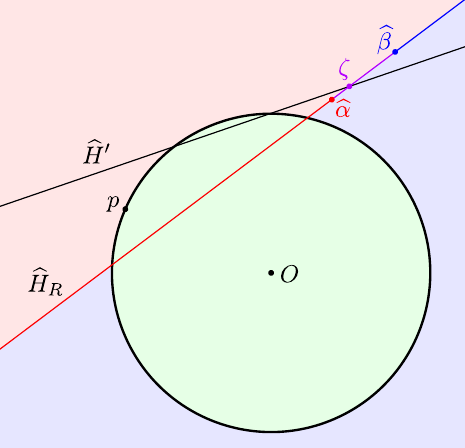}\qquad
    \includegraphics[height=70mm]{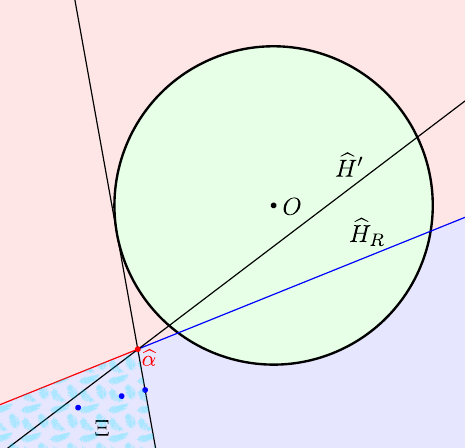}
  \end{center}
\caption{An illustration of the proof of \cref{prop:join-irr}. Red shading indicates areas where rescaled roots are in $\widehat R$, and blue shading indicates areas where rescaled roots are in $\wh\Phi^+\setminus\wh R$. }
\label{fig:JIsareJIs} 
\end{figure}

Say a biclosed set $R\subseteq\Phi^+$ is \dfn{finitely generated} if there exists a finite set $X\subseteq\Phi^+$ such that ${R=\Phi^+\cap\spange X}$. 
It follows from \cite[Lemma 1.7]{dyer:2019weak} that finite and cofinite biclosed sets are finitely generated. Now let
$R$ be a biclosed set that is neither finite nor cofinite. If
$\Phi_R$ has rank 2 and $\partial R$ is cofinite, then $R$ is finitely
generated. In every other case, $R$ is not finitely generated. Note,
however, that $R$ is always a join of finitely generated biclosed
sets. Since finitely generated biclosed sets are
(lattice-theoretically) compact, this implies that the extended weak
order of $U_3$ is an algebraic lattice (see \cite[Section~1.9]{grazter:1971lattice} for definitions of compact and algebraic). 

A biclosed set $R$ is called a \dfn{complete join-irreducible} if for
every join representation $R = \bigvee_{i\in I} R_i$, there exists
some $i\in I$ such that $R=R_i$. Equivalently, $R$ is a complete
join-irreducible if there exists a biclosed set $R_*\subsetneq R$ such
that every biclosed set $R'$ satisfying $R'\subsetneq R$ also satisfies $R'\subseteq R_*$.

\begin{proposition}\label{prop:join-irr}
The following is an exhaustive list of the
complete join-irreducibles: 
\begin{itemize}
\item For each non-identity element
  $w\in U_3$, the set $\Inv(w)$ is complete join-irreducible. 
	
\item For each point $p\in \partial\DD$, let $\wh H$ be the tangent
  line through $p$ in $\AA$, and let $H$ be the linear hyperplane in $V$
  containing $\wh H$. Let $\cH$ be the half-space in $V\setminus H$
  that contains no points of $\DD$. If $H\cap \Phi$ is a rank 2 subsystem, then for each proper cofinite
  biclosed subset $\partial R \subsetneq H\cap \Phi^+$, the biclosed
  set $(\cH\cap \Phi^+) \sqcup \partial R$ is a complete
  join-irreducible.
	
\item For each pair of distinct points $p,q\in \partial\DD$, let $\wh
  H$ be the line through $p$ and $q$ in $\AA$, and let $H$ be the linear
  hyperplane in $V$ containing $\wh H$. If $H\cap \Phi$ is a rank 2
  subsystem, then for each of the two half-spaces $\cH$ that is a
  component of $V\setminus H$, and for each proper cofinite biclosed
  subset $\partial R \subsetneq H\cap \Phi$, the biclosed set
  $(\cH\cap \Phi^+) \sqcup \partial R$ is a complete join-irreducible.
\end{itemize}
\end{proposition} 

\begin{proof}
    We first show that the biclosed sets $R$ listed are each complete join-irreducible. Note that each $R$ in the list covers a unique element $R_*$ by \Cref{coverrelations}.  If $R$ is finite, then the interval $[\varnothing, R]$ in extended weak order is a finite poset, so the fact that $R$ covers a unique element implies it is a complete join-irreducible. Let $\alpha$ be the unique root in $R\setminus R_*$.  In general, we must check that for any biclosed $R'\subseteq R$, either $\alpha\not\in R'$ or $R'=R$. 
    
    So now assume $R$ is neither finite nor cofinite. Let $H_R, \cH_R, \Phi_R, \partial R$ be as above. Let $R'\subsetneq R$ be an arbitrary biclosed set. Then $R'\cap \Phi_R$ is a biclosed set in $\Phi_R^+$ that is contained in $\partial R$. By construction, the unique set $C$ that is biclosed in $\Phi_R^+$ with $\{\alpha\} \subseteq C \subseteq \partial R$ is $C=\partial R$. Hence, either $\alpha\not\in R' \cap \Phi_R$ or else $R'\cap \Phi_R = \partial R$. In the first case, we are done. In the second case, consider a weakly separating hyperplane $H'$ for $R'$. We will show that $H'=H_R$; together with the fact that $R'\cap \Phi_R=\partial R$, this will imply $R=R'$. 
    
    Assume for contradiction that $H'\neq H_R$. Then $\wh H'$ and $\wh H_R$ intersect between $\wh \alpha$ and $\wh \beta$, where $\beta\in \Phi_R^+$ is the unique root such that $\partial R \cup \{\beta\}$ is biclosed in $\Phi_R^+$. Call this intersection point $\zeta$; we may have $\zeta=\wh \alpha$ or $\zeta=\wh \beta$. Let $\cH'$ be the component of $V \setminus \cH'$ containing roots of $R'$. Since $\partial R\subseteq R'$ and $\partial R$ is cofinite, the intersection $\DD \cap \wh H_R$ must be contained in $\cH'$. (See the left side of \Cref{fig:JIsareJIs}.) Then there exists a point $p\in \cH'\cap\partial\DD \setminus (\cH_R\cup H_R)$. By \Cref{lem:limitpoints}, there are rescaled roots in any neighborhood of $p$; in particular, there is some root $\gamma \in \cH'\setminus \cH_R$, contradicting the fact that $R'\subseteq R$.
    This shows that $H'=H_R,$ so $R=R'$.

    Now we check that the remaining biclosed sets are not complete join-irreducible. Each cofinite biclosed set covers two or three others, so cofinite biclosed sets are not complete join-irreducibles. So we may assume $R$ is a biclosed set that is neither finite nor cofinite and is not listed above. Again define $H_R,\cH_R,\Phi_R,\partial R$ as above. For $R$ to have a unique lower cover, examination of \Cref{coverrelations} shows that we need $\partial R$ to be either proper and cofinite (in which case $R$ is in the list above) or nonempty and finite. So we must check that if $\partial R$ is nonempty and finite, and $R$ covers $R_*=R\setminus \{\alpha\}$, then there exists some $R'\subsetneq R$ such that $\alpha\in R'$. To do so, let $\ell$ be the tangent line to $\partial\DD$ passing through $\wh\alpha$ that intersects $\partial\DD$ in $\cH_R$. There is a unique closed  
    region $\Xi$ bounded by $\ell$ and $\wh H_R$ and containing no points of $\DD$ (filled with tiny light blue fern leaves on the right side of \Cref{fig:JIsareJIs}). The region $\Xi$ has compact intersection with $\conv(\wh\alpha_1,\wh\alpha_2,\wh\alpha_3)$ and (by \Cref{lem:limitpoints}) contains no accumulation points of $\wh\Phi^+$. Hence, there are finitely many rescaled roots in $\Xi$. It follows that there exists a line $\wh H'\neq \wh H_R$ passing through $\wh \alpha$, intersecting $\DD$, and bounding a half-space $\cH'$ such that $\cH'\cap \Xi=\varnothing$. In particular, we have the containment $\cH'\cap \Phi^+\subseteq R$. Then $R'=\{\alpha\}\cup(\cH'\cap \Phi^+)$ is a biclosed set with the property that $\alpha\in R'$ and $R'\subseteq R$. By an application of \Cref{lem:limitpoints} similar to above, $R'$ is a proper subset of $R$. Hence, $R$ is not a complete join-irreducible.
\end{proof}

Note that every complete join-irreducible is finitely generated. The
remaining finitely generated biclosed sets are \begin{itemize}
\item $\varnothing$,
\item biclosed sets $R$
such that $\Phi_R$ has rank 2 and $\partial R = \Phi_R$,
\item proper, cofinite biclosed sets, 
\item $\Phi^+$.
\end{itemize} 
In these
cases, the biclosed set is the join of zero, two, two, and three distinct complete
join-irreducibles, respectively.  

\section{The Join Operation}\label{sec:closure}

This section is devoted to proving \cref{thm:main2}.

The \dfn{convex closure} of a set $X\subseteq\Phi^+$ is the set $\Phi^+\cap\spange X$. The \dfn{2-closure} of $X$, denoted $\overline{X}$, is the unique inclusion-minimal closed subset of $\Phi^+$ containing $X$. In general, the 2-closure of $X$ is contained in the convex closure of $X$.

The next result, which follows immediately from \Cref{thm:main} and
\cite[Corollary 2.36]{LabbeThesis}, states that we can compute the join of a collection of biclosed sets in the extended weak order by taking the convex closure of the union of the collection. 

\begin{proposition}\label{prop:convexjoin}
  Let $\{R_i \}_{i\in I}$ be a collection of biclosed sets in
  $\Phi$. Then the join in extended weak order exists and is given by
	\[ \bigvee_{i\in I} R_i = \Phi^+\cap \spange\bigcup_{i\in I}R_i. \]
\end{proposition}

In terms of rescaled roots in the affine patch
$\AA$, the join operation is given by
\[ \wh{\bigvee_{i \in I} R_i} = \wh\Phi^+ \cap
\conv\bigcup_{i\in I} \wh R_i. \] 

One might try to replace the convex closure in \cref{prop:convexjoin} with the weaker 2-closure. Our goal in this section is to prove \cref{thm:main2}, which states that one can indeed do this. 
This theorem strengthens \cref{prop:convexjoin} and settles a conjecture of Dyer in the special case of the Coxeter group~$U_3$. 
	
\begin{remark}
Given our use of parabolic biclosed sets, one might wonder if the
  join of a collection of biclosed sets is given by the minimal parabolic closed set containing
  their union. If this were so, then the join operation would have a
  simple description in terms of snakes. Unfortunately, this is not
  the case: the parabolic closure of a union of biclosed sets is often
  not biclosed.
\end{remark}

\begin{lemma}\label{lem:convextriangles}
  Suppose $X$ is a subset of $\mathbb{R}^2$. For each $a\in X$, we have 
  \[ \conv(X) = \bigcup_{x,y\in X}\conv\{a,x,y\}. \]
\end{lemma}

\begin{proof}
  Let $z\in \conv(X)$. 
  By Carath\'eodory's Theorem, there is a triangle $\conv\{w,x,y\}$ containing $z$ such that $w,x,y\in X$.
  Consider the set of four points $\{a,w,x,y\}$. 
  Observe that the union of the three triangles generated by $a$ and two of the vertices in $\{w,x,y\}$ covers the triangle $\conv\{w,x,y\}$. 
  Hence, $z$ is in one of these triangles.
\end{proof}

Before computing arbitrary joins in extended weak order, we first consider a special case.

\begin{lemma}\label{lem:2closure_join}
    If $\{R_i\}_{i\in I}$ is a collection of biclosed subsets of $\Phi^+$ such that $\bigcap_{i\in I} R_i$ contains a simple root, then the join of $\{R_i\}_{i\in I}$ in extended weak order is the 2-closure of its union. 
\end{lemma}

\begin{proof}
  Without loss of generality, we may assume that $\alpha_1\in\bigcap_{i\in I}R_i$.
  Let $\cU=\bigcup_{i\in I} R_i$.
  Let $\cZ=\Phi^+\cap\spange \,\cU$ so that $\cZ$ is the join of $\{R_i\}_{i\in I}$ in extended weak order by \cref{prop:convexjoin}. 
  Since convex closure is stronger than 2-closure, we know that $\cZ$ contains $\overline{\cU}$. 
  We wish to show these sets are equal.
  Assume to the contrary that $\cZ\setminus\overline{\cU}$ is nonempty, and choose a root $\gamma\in \cZ\setminus \overline{\cU}$ of minimum height.

  By \cref{lem:convextriangles}, there exist roots $\beta_1,\beta_2\in \cU$ such that $\wh{\gamma}$ is in the triangle $\conv\{\wh{\alpha}_1,\wh{\beta}_1,\wh{\beta}_2\}$. 
  Choose $X,Y\in \{R_i\}_{i\in I}$ such that $\beta_1\in X$ and $\beta_2\in Y$.  
  Since $\conv\{\wh{\alpha}_1,\wh{\beta}_1,\wh{\beta}_2\}$ is not contained in either of the convex sets $\wh{X},\wh{Y}$, we have $X\cap \{\alpha_1,\beta_1,\beta_2\} = \{\alpha_1,\beta_1\}$ and $Y\cap\{\alpha_1,\beta_1,\beta_2\}=\{\alpha_1,\beta_2\}$.
  Let $L_X,L_Y$ be weakly separating lines for $\wh{X},\wh{Y}$, respectively. 

 Consider the line $L_1$ through $\wh{\gamma}$ and $\wh{\beta}_1$.
 Divide this line into two rays $\rho_1,\rho_1^{\prime}$ meeting at $\wh{\gamma}$, and suppose $\wh{\beta}_1$ is in $\rho_1$. 
 By \cref{cor:rank2_subsystem}, the line $L_1$ must have a nontrivial intersection with $\DD$.
 Since $\DD$ is a convex set, $\DD\cap L_1$ is either a point or a line segment. 
 Moreover, the intersection $\DD\cap L_1$ is contained in either $\rho_1$ or $\rho_1^{\prime}$ since $\wh{\gamma}$ is not in $\DD$.

 {\bf Case 1:} Suppose that the intersection $\DD\cap L_1$ is contained in $\rho_1^{\prime}$, as shown in \cref{fig:triangles_case_1}. Let $\alpha$ be the relatively simple root of the rank~$2$ subsystem $\Phi'$ containing $\beta_1$ and $\gamma$ such that $\langle\alpha,\gamma\rangle<0$.
 By \cref{lem:rank2geom}, the rescaled roots $\wh{\alpha}$ and $\wh\gamma$ are on opposite sides of $\DD\cap L_1$. 
 By our assumption on $\DD\cap L_1$, this implies that $\wh{\alpha}\in \rho_1^{\prime}$.
 Since $\gamma$ is not relatively simple in $\Phi'$, $\height{\alpha}<\height{\gamma}$.

 The ray $\rho_1^{\prime}$ intersects the boundary of the triangle $\conv\{\wh{\alpha}_1,\wh{\beta}_1,\wh{\beta}_2\}$ at a point $p$ on the edge between $\wh{\alpha}_1$ and $\wh{\beta}_2$. 
 Since $L_Y$ weakly separates $\wh{\gamma}$ from the edge between $\wh{\alpha}_1$ and $\wh{\beta}_2$, it weakly separates $\wh{\gamma}$ from $p$. 
 Hence, $\wh{\alpha}$ is either in the triangle $\conv\{\wh{\alpha}_1,\wh{\beta}_1,\wh{\beta}_2\}$ or is weakly separated from $\wh{\gamma}$ by $L_Y$. 
 In either case, $\alpha$ is in the convex closure of $X\cup Y$. 
 By the minimality on the height of $\gamma$, we deduce that $\alpha$ is in $\overline{\cU}$. 
 Since $\wh{\gamma}$ is between $\wh{\beta}_1$ and $\wh{\alpha}$, we conclude that $\gamma$ is also in $\overline{\cU}$, which is a contradiction.

 {\bf Case 2:} Now suppose that $\DD\cap L_1$ is in $\rho_1$, as shown in \cref{fig:triangles_case_2}. Consider the line $L_2$ through $\wh{\gamma}$ and $\wh{\beta}_2$. 
 Divide this line into rays $\rho_2,\rho_2^{\prime}$ meeting at $\wh{\gamma}$ such that $\wh{\beta}_2$ is in $\rho_2$. 
 If $\DD\cap L_2$ is contained in $\rho_2^{\prime}$, then we reach a contradiction as in the previous case. 
 Hence, we may assume $\DD\cap L_2$ is contained in $\rho_2$.

 Let $M$ be one of the two tangent lines to $\partial\DD$ going through $\wh{\gamma}$. 
 If $M$ intersects the interior of the line segment between $\wh{\beta}_1$ and $\wh{\beta}_2$, then this tangent line separates $\rho_1$ and $\rho_2$. 
 But $\DD$ only appears on one side of $M$, so it cannot meet both $\rho_1$ and $\rho_2$, which is a contradiction. Hence, the line $M$ leaves the triangle $\conv\{\wh{\alpha}_1,\wh{\beta}_1,\wh{\beta}_2\}$ through its other two sides or its vertices. 
 By the same argument as in the previous case, the entire line $M$ intersected with $\wh{\Phi}^+$ is contained in the convex closure of $X\cup Y$. 
 We may choose $M$ to be the tangent line to $\partial\DD$ through $\wh{\gamma}$ that corresponds to a parabolic rank $2$ subsystem in which $\gamma$ is not simple by \cref{prop:rank2_parabolics}. 
 Thus, the extreme roots on $M$ are each of lower height than $\gamma$, so they are contained in $\overline{\cU}$ by the minimality hypothesis on $\gamma$. 
 But then the entire line $M$ intersected with $\Phi^+$ is contained in $\overline{\cU}$, which is a contradiction.
\end{proof}

\begin{figure}[]
  \begin{center}
    \includegraphics[height=97.938mm]{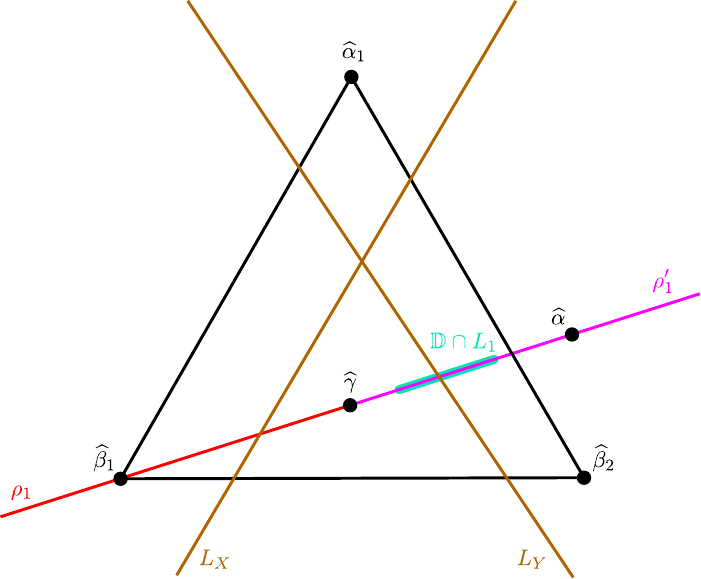}
  \end{center}
\caption{An illustration of the proof of \cref{lem:2closure_join} in the case where $\mathbb D\cap L_1$ is contained in the ray $\rho_1'$.}
\label{fig:triangles_case_1}  
\end{figure} 

\begin{figure}[]
  \begin{center}
    \includegraphics[height=97.938mm]{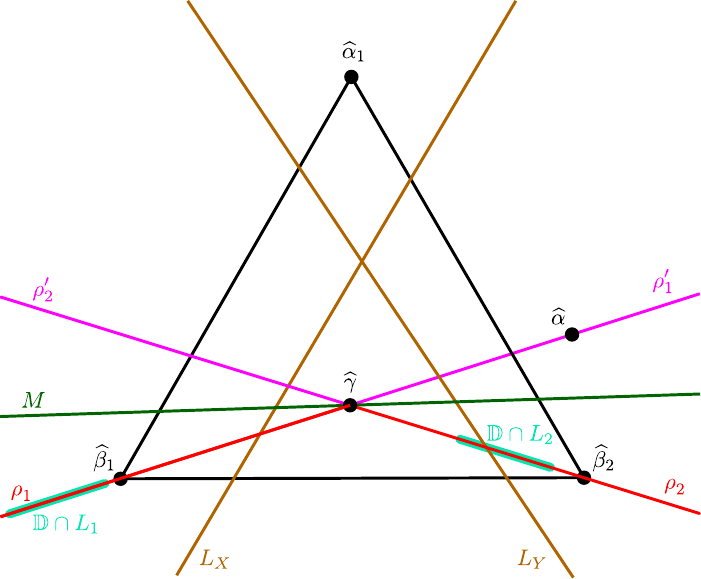}
  \end{center}
\caption{An illustration of the proof of \cref{lem:2closure_join} in the case where ${\mathbb D\cap L_1\subseteq \rho_1}$ and $\mathbb D\cap L_2\subseteq \rho_2$.} 
\label{fig:triangles_case_2}  
\end{figure} 

\begin{lemma}\label{lem:closeDelta}
    The parabolic 2-closure of $\Delta$ is $\Phi^+$.
\end{lemma}

\begin{proof}
    Suppose not, and choose a positive root $\gamma$ of minimum depth such that $\gamma$ is not in the parabolic 2-closure of $\Delta$.
    Since $\gamma\notin\Delta$ there is a rank 2 parabolic subsystem $\Phi'$ containing $\gamma$ for which $\gamma$ is not relatively simple by \cref{prop:rank2_parabolics}.
    The two simple roots $\beta_1,\beta_2$ of $\Phi'$ have lower depth than $\gamma$, so they are in the parabolic 2-closure of $\Delta$ by assumption.
    But then $\Phi'^+$ is in the parabolic 2-closure of $\Delta$, contradicting the assumption on $\gamma$.
\end{proof}

\begin{proof}[Proof of \cref{thm:main2}]
Let $\{R_i\}_{i\in I}$ be a collection of biclosed subsets of $\Phi^+$. We may discard the empty set from this collection without affecting the join or the 2-closure of the union; hence, we assume that all of the sets in the collection are nonempty.
 If $\cU=\bigcup_{i\in I}R_i$ contains all three simple roots, then $\overline{\cU} = \Phi^+$ by \cref{lem:closeDelta}, and $\Phi^+$ is the join of $\{R_i\}_{i\in I}$ by \cref{prop:convexjoin}. 
 Each nonempty biclosed set contains a simple root by \cref{cor:bicsimple}.
 If each set $R_i$ contains the same simple root, then \cref{lem:2closure_join} completes the proof.
 It remains to consider the case where $\cU$ contains exactly two simple roots, say $\alpha_1$ and $\alpha_2$.
 
 Let $\Phi_{12}$ be the standard parabolic subsystem generated by $\{\alpha_1,\alpha_2\}$. 
 Then $\Phi_{12}^+$ is a biclosed set that is contained in the 2-closure of $\cU$, so we have
 \[ \overline{\cU} = \overline{\cU\cup \Phi_{12}^+} = \overline{\bigcup_{i\in I} \overline{R_i\cup \Phi_{12}^+}}. \]
 Since each set $R_i$ shares a simple root with $\Phi_{12}^+$, we have $\overline{R_i\cup \Phi_{12}^+}=R_i\vee \Phi_{12}^+$ by \cref{lem:2closure_join}. 
 But then $\{R_i\vee \Phi_{12}^+\}_{i\in I}$ is a collection of biclosed sets that share a common simple root, so we can invoke \cref{lem:2closure_join} again to find that
 \[ \overline{\cU} = \overline{\bigcup_{i\in I}(R_i\vee \Phi_{12}^+)} = \bigvee_{i\in I} (R_i\vee \Phi_{12}^+) \supseteq \bigvee_{i\in I} R_i. \]
We already know by \cref{prop:convexjoin} that $\overline{\cU}\subseteq\bigvee_{i\in I}R_i$, so this completes the proof. 
\end{proof}

\section*{Acknowledgments}

This project grew out of discussions at a workshop on Lattice Theory
at the Banff International Research Station (BIRS) in Banff, Canada in
January 2025.  We are grateful to BIRS for providing a very
stimulating environment. We thank the
other participants of the workshop for their conversations, especially Oliver Pechenik and Nathan Reading. We are particularly grateful to Emily Barnard,
Cesar Ceballos, Osamu Iyama, and Nathan Williams for organizing the workshop.  And finally, we would like to thank the anonymous referee 
for their detailed and helpful comments.

\bibliographystyle{alpha}
\bibliography{biclosed_bib}

\end{document}